\documentclass[smallextended]{svjour3}[]

\usepackage[%
  font=small,
  labelfont=bf,
  tableposition=top
]{caption}
\usepackage{array}
\usepackage{mdwmath}
\usepackage{mdwtab}
\usepackage{subfloat}
\usepackage{eqparbox}
\usepackage{color}
\usepackage{graphicx}
\usepackage{slashbox}
\usepackage{subfig}
\usepackage{multirow}
\usepackage{amsmath, amssymb}
\usepackage{verbatim}
\usepackage{url}

\usepackage{booktabs}
\usepackage{algorithm,algorithmic}

\newcommand{\ignore}[1]{}
\DeclareMathOperator*{\argmin}{arg\,min}

\def\norm2#1{\|#1\|_2}

\def\mplus{\mathrel{%
  \ooalign{\raise.29ex\hbox{$\scriptscriptstyle\mathbf{+}$}\cr}}}
  
 \def\mminus{\mathrel{%
  \ooalign{\raise.29ex\hbox{$\scriptscriptstyle\mathbf{-}$}\cr}}}

\begin{document}
\sloppy
%
\title{A Second-Order Method for Strongly Convex $\ell_1$-Regularization Problems}

\author{Kimon Fountoulakis and Jacek Gondzio\thanks{J. Gondzio is supported by EPSRC Grant EP/I017127/1}}

\institute{ 
Kimon Fountoulakis
  \at School of Mathematics and Maxwell Institute,
      The University of Edinburgh,
      Peter Guthrie Tait Road, Edinburgh EH9 3FD,
      United Kingdom.
      \\\email{K.Fountoulakis@sms.ed.ac.uk}
      \\Tel.: +44 131 650 5083 
 \and 
Jacek Gondzio
  \at School of Mathematics and Maxwell Institute,
      The University of Edinburgh,
      Peter Guthrie Tait Road, Edinburgh EH9 3FD,
      United Kingdom.
      \\\email{J.Gondzio@ed.ac.uk}
      \\Tel.: +44 131 650 8574, Fax: +44 131 650 6553
}

\maketitle

\begin{abstract}
In this paper a robust second-order method is developed for the solution of strongly convex $\ell_1$-regularized problems. 
The main aim is to make the proposed method as inexpensive as possible, while even difficult problems can be efficiently solved.
The proposed approach is a primal-dual Newton Conjugate Gradients (pdNCG) method. 
Convergence properties of pdNCG are studied and worst-case iteration complexity is established. 
Numerical results are presented on synthetic sparse least-squares problems and real world machine learning problems.



\end{abstract}

\section{Introduction}
We are concerned with the solution of the following optimization problem
\begin{equation}\label{prob1}
\mbox{minimize} \  f_{\tau}(x) := \tau \|x\|_1 + \varphi(x),
\end{equation}
where $x\in\mathbb{R}^m$, $\tau > 0$ and $\|\cdot\|_1$ is the $\ell_1$-norm. The following three assumptions are made.
\begin{itemize}
\item The function $\varphi(x)$ is twice differentiable, and 
\item strongly convex everywhere, which implies that at any $x$ its second derivative $\nabla^2 \varphi(x)$ is uniformly bounded
\begin{equation}\label{bd15}
\lambda_m I_m \preceq \nabla^2 \varphi(x) \preceq \lambda_1 I_m ,
\end{equation}
with $0 < \lambda_m \le \lambda_1$, where $I_m$ is the $m\times m$ identity matrix.
\item The second derivative of $\varphi(x)$ is Lipschitz continuous
\begin{equation}\label{bd44}
\|\nabla^2 \varphi(y) - \nabla^2 \varphi(x)\| \le L_\varphi \|y-x\|,
\end{equation}
for any $x, y$, where $L_\varphi\ge0$ is the Lipschitz constant, $\|\cdot\|$ represents the Euclidean distance for vectors
and the spectral norm for matrices.
\end{itemize}

A variety of problems originating from the ``new" economy including Big-Data \cite{peterbigdata},
Machine Learning \cite{machineOpti} and Regression \cite{IEEEhowto:Tibshirani} problems to mention a few can be
cast in the form of \eqref{prob1}. 
Such problems usually consist of large-scale data, which frequently impose restrictions on methods that have been so far employed. 
For instance, the new methods have to be memory efficient and ideally, within seconds they should offer noticeable progress in reducing 
the objective function. First-order methods meet some of these requirements. 
They avoid matrix factorizations, which implies low memory requirements, additionally, 
they sometimes offer fast progress in the initial stages of optimization. 
Unfortunately, as demonstrated by numerical experiments presented later in this paper, first-order methods miss essential information about the conditioning of the problems, which might result in slow practical convergence.
The main advantage of first-order methods, which is to rely only on simple gradient or coordinate updates becomes their essential weakness.

We do not think this inherent weakness of first-order methods can be remedied.
For this reason, in this paper, a second-order method is used instead, i.e.,, a primal-dual Newton Conjugate Gradients. The optimization community seems to consider the second-order methods to be rather expensive.
The main aim in this paper is to make the proposed method as \textit{inexpensive} as possible, while even complicated problems can be efficiently solved. 
To accomplish this, pdNCG is used in a matrix-free environment i.e.,, Conjugate Gradients is used to compute inexact Newton directions.
No matrix factorization is performed and no excessive memory requirements are needed. Consequently, the main drawbacks of Newton method are removed, while
at the same time their fast convergence properties are provably retained.
In order to meet this goal, the $\ell_1$-norm is approximated by a smooth function, which has derivatives of all degrees. Hence, problem \eqref{prob1} is replaced by
\begin{equation}\label{prob2}
 \mbox{minimize} \  f_{\tau}^{\mu}(x) := \tau \psi_{\mu}(x) + \varphi(x).
\end{equation}
where $\psi_{\mu}(x)$ denotes the smooth function, which substitutes the $\ell_1$-norm and $\mu$ is a parameter that controls the quality of approximation.
Smoothing will allow access to second-order information and essential curvature information will be exploited. 

On the theoretical front we show that the analysis of pdNCG can be performed in a variable metric using an important property of CG.
The variable metric is the standard Euclidean norm scaled by an approximation of the second-order derivative at every iteration of pdNCG. Based on the variable metric 
we give a complete analysis of pdNCG, i.e., proof of global convergence, 
global and local convergence rates, local region of fast convergence rate and worst-case iteration complexity.

In what follows in this section we give a brief introduction of the smoothing technique. In Section \ref{sec2}, necessary basic results 
are given, which will be used to support theoretical results in Section \ref{sec:ncg}. In Section \ref{sec:pdnewt}, the proposed pdNCG method is described in details. 
In Section \ref{sec:ncg}, the convergence analysis and worst-case iteration complexity of pdNCG is studied. 
In Section \ref{sec:nexp}, numerical results are presented. 

\subsection{Pseudo-Huber regularization}
The non-smoothness of the $\ell_1$-norm prevents a straightforward application of the second-order method to problem \eqref{prob1}.
In this subsection, we focus
on approximating the non-smooth $\ell_1$-norm by a smooth function.
To meet such a goal, the first-order methods community replaces 
the $\ell_1$-norm with the so-called Huber penalty function $\sum_{i=1}^m \phi_\mu(x_i)$  \cite{IEEEhowto:Nesta},
where
\begin{equation*}
  \phi_\mu(x_i) = \left\{ 
  \begin{array}{l l}
    \frac{1}{2}\frac{x_i^2}{\mu}, & \quad \text{if $|x_i|\le \mu$}\\
    |x_i| - \frac{1}{2}\mu, & \quad \text{if $|x_i|\ge \mu$}
  \end{array} \right.
  \ \ i=1,2,\ldots ,m
\end{equation*}
and $\mu>0$.
The smaller the parameter $\mu$ of the Huber function is, the better the function approximates the $\ell_1$-norm.
Observe that the Huber function is only first-order differentiable, therefore, this approximation trick is not applicable to second-order methods. Fortunately, there is a smooth
version of the Huber function, the pseudo-Huber function, which has derivatives of all degrees \cite{Hartley2004}. 
The pseudo-Huber function parameterized with $\mu>0$ is
\begin{equation}\label{pseudoHuber}
 \psi_\mu(x) = \sum_{i=1}^m \Big((\mu^2+{x_i^2})^{\frac{1}{2}} - \mu\Big).
\end{equation}
A comparison of the three functions $\ell_1$-norm, Huber and Pseudo-Huber function can be seen in Figure \ref{fig1}.
\begin{figure}%
\centering
\subfloat[$\ell_1$-norm, Huber, Pseudo-Huber functions]{%
\label{fig1a}%
\includegraphics[width=59mm,height=50mm]{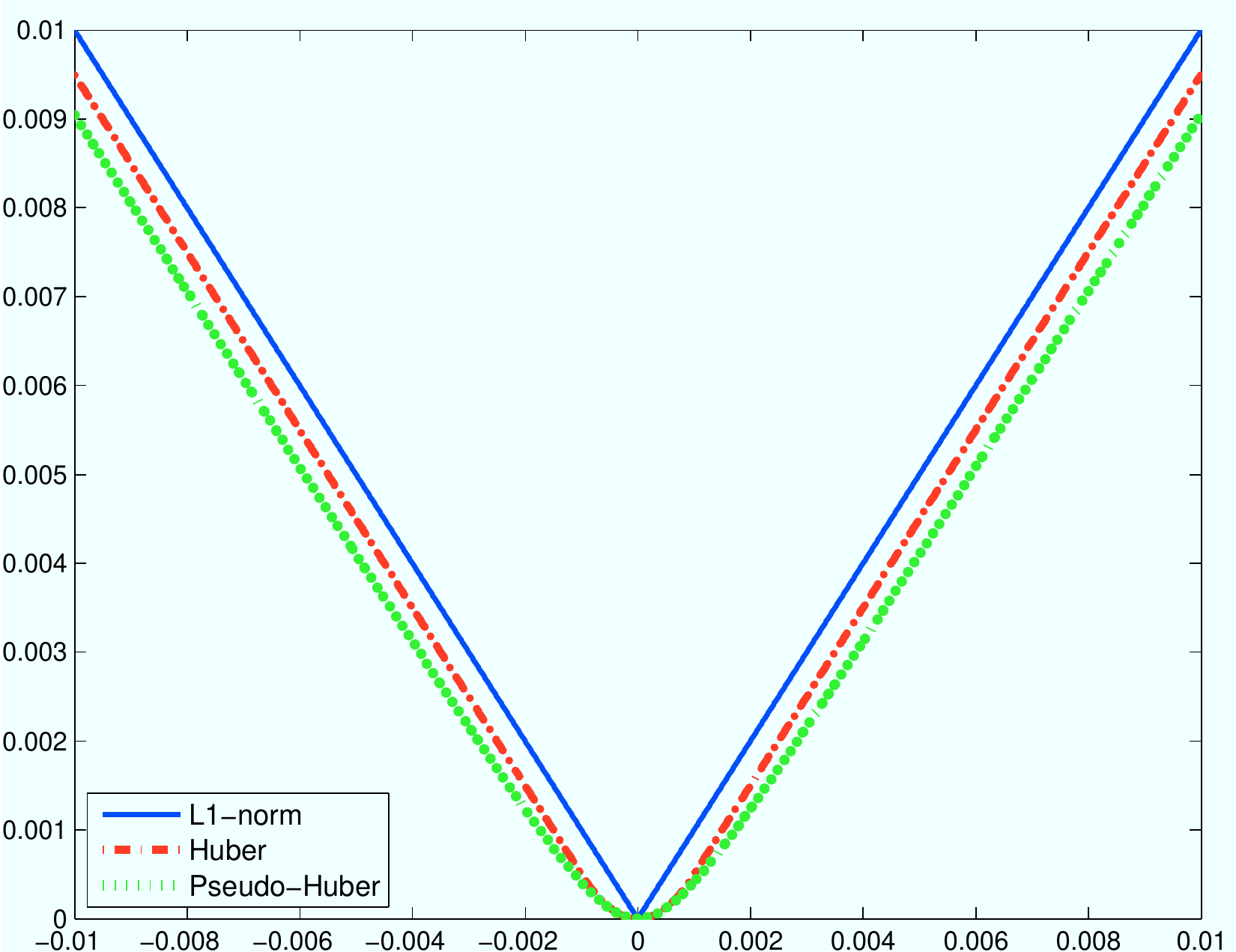}}
\subfloat[Pseudo-Huber function for $\mu\to 0$ ]{%
\label{fig2a}%
\includegraphics[width=59mm,height=50mm]{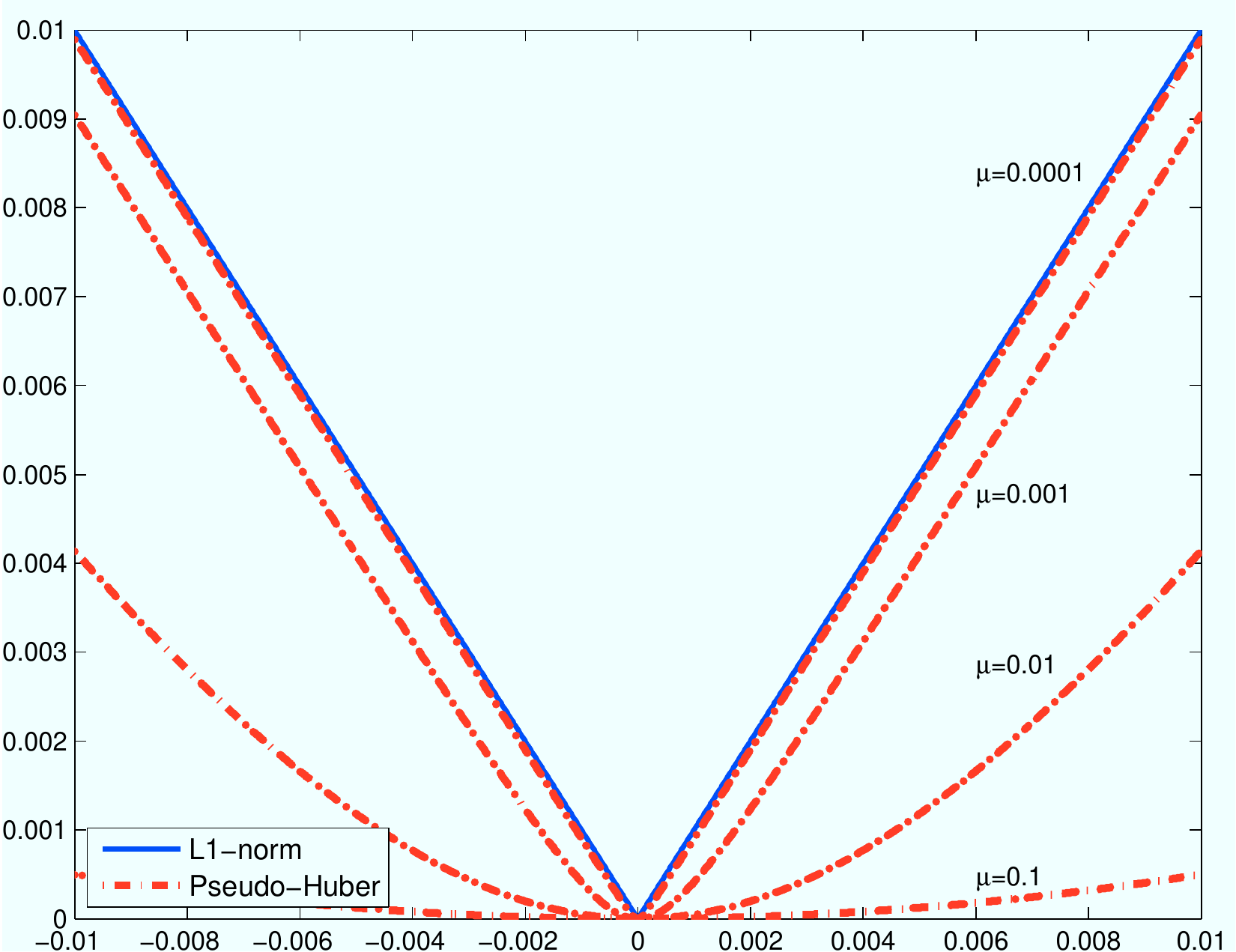}}\\
\caption{Comparison of the approximation functions, Huber and pseudo-Huber, with the $\ell_1$-norm in one dimensional space. \textbf{Fig.1a} shows
the quality of approximation for the Huber and pseudo-Huber functions. \textbf{Fig.1b} shows how pseudo-Huber function converges to the 
$\ell_1$-norm as $\mu\to 0$}
\label{fig1}%
\end{figure}

The advantages of such an approach are listed below.
\begin{itemize}
  \item Availability of second-order information owed to the differentiability of the pseudo-Huber function.
  \item Opening the door to using iterative methods to compute descent directions, which take into account the curvature of the problem, such as CG.
  
\end{itemize}
There is an obvious cost that comes along with the above benefits, and that is the approximate nature of the pseudo-Huber function.
There is a concern that in case that a very accurate solution is required, the pseudo-Huber function may be unable to deliver it. In theory, since the
quality of the approximation is controlled by parameter $\mu$ in (\ref{pseudoHuber}), see Figure \ref{fig1}, the pseudo-Huber function can recover any level of accuracy
under the condition that sufficiently small $\mu$ is chosen. The reader is referred to \cite{pertanal} for a perturbation analysis when the $\ell_1$-norm is replaced with the Pseudo-Huber function. 
In practise a very small parameter $\mu$ might worsen the conditioning of the linear algebra of the solver. However, we shall provide numerical evidence that even when $\mu$ is set to small values, 
the proposed method behaves well and remains very efficient.


\section{Preliminaries}\label{sec2}
The $\|\cdot\|_\infty$ denotes the infinity norm.
The operator $diag(\cdot)$ takes as input a vector and creates a diagonal matrix with the input vector on the diagonal. The operator $[\cdot]_{ij}$ returns the element at row $i$ and column $j$ of the input matrix, similarly, the operator $[\cdot]_i$ returns the element at position $i$ of the input vector.

\subsection{Properties of pseudo-Huber function}
The gradient of the pseudo-Huber function $\psi_\mu(x)$ in (\ref{pseudoHuber}) is given by
\begin{equation}\label{gradpsi}
  \nabla \psi_\mu(x) = \Big[x_1\Big(\mu^2+ {x_1^2}\Big)^{-\frac{1}{2}},\ldots ,x_m\Big(\mu^2+ {x_m^2}\Big)^{-\frac{1}{2}}\Big],
\end{equation}
and the Hessian is given by
\begin{equation}\label{hespsi}
  \nabla^2 \psi_\mu(x) = \mu^2diag\Big(\Big[(\mu^2 + {x_1^2} )^{-\frac{3}{2}},\ldots ,(\mu^2 + {x_m^2} )^{-\frac{3}{2}} \Big]\Big).
\end{equation}
The next lemma guarantees that the Hessian of the pseudo-Huber function $\psi_\mu(x)$ is bounded.
\begin{lemma}\label{lem:4}
The Hessian matrix $\nabla^2 \psi_\mu(x)$ satisfies
\begin{equation*}
 0I \prec \nabla^2 \psi_\mu(x) \preceq \frac{1}{\mu}I
\end{equation*}
where $I$ is the identity matrix in appropriate dimension.
\end{lemma}
\begin{proof}
The result follows easily by observing that $0 < (\mu^2+ {x_i^2} )^{-\frac{3}{2}}\le 1/\mu^3 $ for any $x_i$, $i=1,2\ldots ,m$.
The proof is complete.
\end{proof}
The next lemma shows that the Hessian matrix of the pseudo-Huber function is Lipschitz continuous. 
\begin{lemma}\label{lem:5}
The Hessian matrix $\nabla^2 \psi_\mu(x)$ is Lipschitz continuous
\begin{equation*}
  \|\nabla^2 \psi_\mu(y) - \nabla^2 \psi_\mu(x)\| \le  \frac{1}{\mu^2}\|y-x\|.
\end{equation*}
\end{lemma}
\begin{proof}

\begin{align}
\|\nabla^2 \psi_\mu(y) - \nabla^2 \psi_\mu(x)\| &= \Big\|\int_0^1 \frac{d \nabla^2 \psi_\mu(x + s(y-x))}{ds} ds\Big\| \nonumber \\
 								       & \le \int_0^1 \Big\| \frac{d \nabla^2 \psi_\mu(x + s(y-x))}{ds}\Big\| ds \label{bd21}
\end{align}
where $\frac{d \nabla^2 \psi_\mu(x + s(y-x))}{ds}$ is a diagonal matrix with each diagonal component, $i=1,2,\ldots m$, given by
\begin{equation*}
\Big[\frac{d \nabla^2 \psi_\mu(x + s(y-x))}{ds}\Big]_{ii} =  \frac{-3\mu^2(x_i + s(y_i-x_i))(y_i-x_i)}{(\mu^2 + {(x_i+s(y_i-x_i))^2})^{\frac{5}{2}}}.
\end{equation*}
Using the previous observation we have that
\begin{align}
\Big\| \frac{d \nabla^2 \psi_\mu(x + s(y-x))}{ds}\Big\| 
 									        &= \displaystyle\max_{i=1,2,\ldots ,m} \Big|\Big[\frac{d \nabla^2 \psi_\mu(x + s(y-x))}{ds}\Big]_{ii}\Big| \label{bd20}
\end{align}
Moreover, we have
\begin{align}
\Big|\Big[\frac{d \nabla^2 \psi_\mu(x + s(y-x))}{ds}\Big]_{ii}\Big| &= \Big| \frac{-3\mu^2(x_i + s(y_i-x_i))(y_i-x_i)}{(\mu^2 + {(x_i+s(y_i-x_i))^2})^{\frac{5}{2}}} \Big| \nonumber\\
                    										&= \Big| \frac{-3\mu^2(x_i + s(y_i-x_i))}{(\mu^2 + {(x_i+s(y_i-x_i))^2})^{\frac{5}{2}}} \Big||(y_i-x_i)| \label{bd17}
\end{align}
where the first absolute value in (\ref{bd17}) has a maximum at $\frac{\mu - 2x_i}{2(y_i-x_i)}$, which gives
\begin{align}\label{bd18}
\Big|\frac{-3\mu^2(x_i + s(y_i-x_i))}{(\mu^2 + {(x_i+s(y_i-x_i))^2})^{\frac{5}{2}}} \Big| \le \frac{48}{25\sqrt{5}\mu^2} < \frac{1}{\mu^2}.
\end{align}
Combining (\ref{bd17}) and (\ref{bd18}) we get
\begin{equation}\label{bd19}
\Big|\Big[\frac{d \nabla^2 \psi_\mu(x + s(y-x))}{ds}\Big]_{ii}\Big| \le \frac{1}{\mu^2}|y_i-x_i|.
\end{equation}
Replacing (\ref{bd19}) in (\ref{bd20}) and using the fact that $\|\cdot\|_\infty \le \|\cdot \|$ we get
\begin{equation*}
\Big\| \frac{d \nabla^2 \psi_\mu(x + s(y-x))}{ds}\Big\| \le \frac{1}{\mu^2}\|y-x\|.
\end{equation*}
Replacing the above expression in (\ref{bd21}) and calculating the integral we arrive at the desired result. The proof is complete.

The next lemma shows that the gradient of the pseudo-Huber function is Lipschitz continuous. 
\begin{lemma}\label{lem:16}
The gradient $\nabla \psi_\mu(x)$ is Lipschitz continuous
\begin{equation*}
  \|\nabla \psi_\mu(y) - \nabla \psi_\mu(x)\| \le  \frac{1}{\mu}\|y-x\|.
\end{equation*}
\end{lemma}
\begin{proof}
Using the fundamental theorem of calculus, like in proof of Lemma \ref{lem:5}, and Lemma \ref{lem:4} it is easy to show the result. 
The proof is complete.
\end{proof}

\end{proof}

\subsection{Properties of function $f_\tau^\mu(x)$}
The gradient of $f_\tau^\mu(x)$ is given by
\begin{equation*}
\nabla f_\tau^\mu(x) = \tau \nabla \psi_\mu(x) + \nabla \varphi(x)
\end{equation*}
where $\nabla \psi_\mu(x)$ has been defined in (\ref{gradpsi}).
The Hessian matrix of $f_\tau^\mu(x)$ is 
\begin{equation*}
\nabla^2 f_\tau^\mu(x) = \tau \nabla^2 \psi_\mu(x) + \nabla^2 \varphi(x).
\end{equation*}
where $\nabla^2 \psi_\mu(x)$ has been defined in (\ref{hespsi}).
Using (\ref{bd15}) and Lemma \ref{lem:4} we get the following bounds on the Hessian matrix of $f_{\tau}^{\mu}(x)$
\begin{equation} \label{bd2}
  \lambda_m I \prec \nabla^2 f_\tau^\mu(x) \preceq \Big(\frac{\tau}{\mu} + \lambda_1\Big)  I,
\end{equation}
where $I$ is the identity matrix in appropriate dimension. 
 \begin{lemma}\label{lem:13}
For any $x$ and $x^*$, the minimizer of $f_\tau^\mu(x)$, the following holds
$$
\frac{1}{2\Big(\frac{\tau}{\mu} + \lambda_1\Big) }\|\nabla f_\tau^\mu(x)\|^2 \le f_\tau^\mu(x) - f_\tau^\mu(x^*) \le \frac{1}{2\lambda_m}\|\nabla f_\tau^\mu(x)\|^2
$$
and
$$
\|x-x^*\|\le \frac{2}{\lambda_m}\|\nabla f_\tau^\mu(x)\|.
$$
\end{lemma}
\begin{proof}
The right hand side of the first inequality is proved on page $460$ of \cite{bookboyd}.
The left hand side of the first inequality is proved by using strong convexity of $f_\tau^\mu(x)$,
$$
f_\tau^\mu(y) \le f_\tau^\mu(x) + \nabla f_\tau^\mu(x)^\intercal (y-x) + \frac{\frac{\tau}{\mu}+\lambda_1}{2}\|y-x\|^2
$$
and defining $\tilde{y}=x-\frac{1}{\frac{\tau}{\mu} + \lambda_1}\nabla f_\tau^\mu(x)$. We get
\begin{equation*}
f_\tau^\mu(x) - f_\tau^\mu(x^*)\ge f_\tau^\mu(x) - f_\tau^\mu(\tilde{y}) \ge \frac{1}{2\Big(\frac{\tau}{\mu} + \lambda_1\Big)}\|\nabla f_\tau^\mu(x)\|^2. 
\end{equation*}
The last inequality is proved on page $460$ of \cite{bookboyd}. The proof is complete.
\end{proof}
The following lemma guarantees that the Hessian matrix $\nabla^2 f_\tau^\mu(x) $ is Lipschitz continuous. In this lemma, $L_\varphi$ is defined in \eqref{bd44}.
\begin{lemma}\label{lem:7}
The function $ \nabla^2 f_{\tau}^{\mu}(x)$ is Lipschitz continuous
\begin{equation*}
  \|\nabla^2 f_{\tau}^{\mu}(y)- \nabla^2 f_{\tau}^{\mu}(x)\| \le  L_{f_\tau^\mu}\|y-x\|,
\end{equation*}
where $L_{f_\tau^\mu}:=\frac{\tau}{\mu^2} + L_\varphi$.
\end{lemma}
\begin{proof}
Using Lemma \ref{lem:5} and (\ref{bd44}) we have
\begin{align}
\|\nabla^2 f_{\tau}^{\mu}(y)- \nabla^2 f_{\tau}^{\mu}(x)\| & \le \tau\|\nabla^2 \psi_\mu(y) - \nabla^2 \psi_\mu(x)\| + \|\nabla^2 \varphi(y) - \nabla^2 \varphi(x)\| \nonumber\\
                                                                                                  & \le \Big(\frac{\tau}{\mu^2} + L_\varphi\Big)\|y-x\|. \nonumber
\end{align}
The Lipschitz constant of $\nabla^2 f_{\tau}^{\mu}(x)$ is therefore $L_{f_\tau^\mu}:=\frac{\tau}{\mu^2} + L_\varphi$.
\end{proof}
The next lemma shows how well the second-order Taylor expansion of $f_\tau^\mu(x)$ approximates the function $f_\tau^\mu(x)$.
\begin{lemma}\label{lem:2}
If $q_\tau^\mu(y)$ is a quadratic approximation of the function $f_\tau^\mu(x)$ at $x$
\begin{equation*}
  q_\tau^\mu(y) := f_\tau^\mu(x) + \nabla f_\tau^\mu(x)^\intercal (y-x) + \frac{1}{2}(y-x)^\intercal\nabla^2 f_\tau^\mu(x)(y-x),
\end{equation*}
then
\begin{equation*}
  | f_\tau^\mu(y) - q_\tau^\mu(y) | \le \frac{1}{6} L_{f_\tau^\mu}\|y-x\|^3.
\end{equation*}
\end{lemma}
\begin{proof}
  Using corollary 1.5.3 in \cite{renegarbook} and Lemma \ref{lem:7} we have
  \begin{align*}
    | f_\tau^\mu(y) - q_\tau^\mu(y) | & \le \|y-x\|^2\int_0^1\int_0^t\|\nabla^2 f_\tau^\mu(x + s(y-x)) - \nabla^2 f_\tau^\mu(x)\| dsdt \\
                                       &\le \|y-x\|^2\int_0^1\int_0^t s  L_{f_\tau^\mu} \|y-x\| dsdt \\
                                       &= \frac{1}{6} L_{f_\tau^\mu} \|y-x\|^3.
  \end{align*}
  The proof is complete.
\end{proof}

\subsection{Alternative optimality conditions}
The first-order optimality conditions of problem \eqref{prob2} are $\nabla f_\tau^\mu(x)=\tau\nabla \psi_\mu(x) + \nabla\varphi(x)=0$. Therefore, one could simply apply a Newton-CG method
in order to find a root of this equation. However, in a series of papers \cite{ctnewtonold,ctpdnewton} it has been noted that the linearization of $\nabla \psi_\mu(x)$ for Newton-CG method might be
a poor approximation of $\nabla \psi_\mu(x)$ close to the optimal solution, hence, the method is misbehaving. This argument is supported with numerical experiments in \cite{ctpdnewton}. it is also worth mentioning that our 
empirical experience confirms the results of the previous paper. To deal with this problem the authors in \cite{ctpdnewton} suggested to solve a reformulation of the optimality conditions which, for the problems
of our interest, is
\begin{align}
\tau y  + \nabla \varphi(x) & = 0 \nonumber \\
D^{-1}y - x& = 0 \label{bd33} \\
\|y\|_\infty \le 1,& \nonumber
\end{align}
where $D$ is a diagonal matrix with components
\begin{equation}\label{bd34}
[D]_{ii} =  (\mu^2 + x_i^2)^{-\frac{1}{2}}\ \forall i=1,2,\ldots,m.
\end{equation}
The idea behind this reformulation is that the linearization of the second equations in \eqref{bd33}, i.e., $y_i/[D]_{ii} - x_i = 0$ $\forall i=1,2,\ldots,m$, is of much better quality than the linearization of $[\nabla \psi_\mu(x)]_i$ for $\mu \approx 0$
and $x_i\approx 0$; a scenario that is unavoidable since for small $\mu$ the optimal solution of \eqref{prob2} is expected to be approximately sparse. To see why this is true, observe
that for small $\mu$ and $x_i\approx 0$, the gradient $[\nabla \psi_\mu(x)]_i$ becomes close to singular and its linearization is expected to be inaccurate. On the other hand, $y_i/[D]_{ii} - x_i$
as a function of $x_i$ is not singular for $\mu\approx 0$ and $x_i\approx 0$, hence, its linearization is expected to be more accurate. Empirical justification for the previous is also given 
in Section $3$ in \cite{ctpdnewton}.

In this paper, we follow the same reasoning and solve \eqref{bd33} instead.

\subsection{Primal-dual reformulation}
In this subsection we show that the alternative optimality conditions \eqref{bd33} correspond to a primal-dual formulation of problem \eqref{prob2}. 
This also explains the primal-dual suffix in the name of the proposed method.

Every $i^{th}$ term $|x_i|_\mu:= (\mu^2 + x_i^2)^{1/2} - \mu$ of the pseudo-Huber function in \eqref{pseudoHuber} approximates the absolute value $|x_i|$. The terms 
$|x_i|_\mu$ $\forall i$ can be obtained by regularizing the dual formulation of $|x_i|$ $\forall i$, i.e.,
\begin{equation*}
|x_i| = \sup_{|y_i| \le 1} x_iy_i \quad \mbox{where} \quad y_i\in\mathbb{R},
\end{equation*}
with the term $\mu(1-y_i^2)^{1/2} - \mu$, to obtain
\begin{equation*}
|x_i|_\mu = \sup_{|y_i| \le 1} x_iy_i + \mu(1-y_i^2)^{\frac{1}{2}} - \mu.
\end{equation*}
Hence, the pseudo-Huber function \eqref{pseudoHuber} has the following dual formulation
\begin{equation*}
 \psi_\mu(x) = \sum_{i=1}^m \sup_{|y_i| \le 1} x_iy_i + \mu(1-y_i^2)^{\frac{1}{2}} - \mu.
\end{equation*}
Therefore, problem \eqref{prob2} is equivalent to the primal-dual formulation
\begin{equation}\label{prob3}
 \min_{x} \sup_{\|y\|_\infty \le 1} \tau (x^\intercal y + \mu \sum_{i=1}^m (1-y_i^2)^{\frac{1}{2}} - m \mu) + \varphi(x),
\end{equation}
where $y\in\mathbb{R}^m$ denotes the dual variables. The first order optimality conditions  of the primal-dual formulation \eqref{prob3}
are
\begin{align}
\tau y  + \nabla \varphi(x) & = 0 \nonumber \\
x_i - \mu y_i(1-y_i^2)^{-\frac{1}{2}}& = 0 \ \forall i=1,2,\ldots,m\label{bd5001} \\
\|y\|_\infty \le 1,& \nonumber
\end{align}
It can be easily shown that the first-order optimality conditions \eqref{bd5001} of the primal-dual formulation \eqref{prob3} are equivalent to \eqref{bd33}.

\subsection{A property of Conjugate Gradients algorithm}\label{subsec:andcg}

The following property of CG is used in the convergence analysis of pdNCG.
\begin{lemma}\label{lem:15}
Let $Ax=b$, where $A$ is a symmetric and positive definite matrix. Furthermore, let us assume
that this system is solved using CG approximately; CG is terminated prematurely at the $i^{th}$ iteration. Then if CG is initialized with the zero solution
the approximate solution $x_i$ satisfies 
$$
x_i^\intercal Ax_i  = x_i^\intercal b.
$$
The same result holds when Preconditioned CG (PCG) is used.
\end{lemma}
\begin{proof}
The following property is shown in proof of Lemma 2.4.1 in \cite{ctkelly}. If CG algorithm is initialised with the zero solution $p_0=0$, then it returns a solution $x_i$, which satisfies
\begin{equation*}
x_i := \argmin_p \{\frac{1}{2} p^\intercal A p- p^\intercal b  \ | \ p\in \mathcal{E}_i\},
\end{equation*}
where
\begin{equation*}
\mathcal{E}_i := \mbox{span}(b, Ab, \ldots , A^{i-1}b).
\end{equation*}
Therefore for every $p\in\mathcal{E}_i$, at $t=0$, we get
\begin{equation*}
\begin{split}
\frac{d (\frac{1}{2} (x_i+tp)^\intercal A (x_i+tp) - (x_i+tp)^\intercal b)}{dt} = (Ax_i - b )^\intercal p = 0.
\end{split}
\end{equation*}
Since, $x_i\in\mathcal{E}_i$, then
\begin{equation}\label{bd76}
(A x_i - b)^\intercal x_i = 0 \Longleftrightarrow \nonumber x_i^\intercal Ax_i  = x_i^\intercal b. 
\end{equation}
This completes the first part. In case that PCG is employed with symmetric positive definite preconditioner $P=EE^\intercal$, then PCG is equivalent to solving approximately
the system $E^{-1}AE^{-\intercal}\xi =E^{-1}b$ using CG and then calculating $x_i = E^{-\intercal}\xi_i$. Therefore, by applying the previous we get that $\xi_i^\intercal E^{-1}AE^{-\intercal}\xi_i =\xi_i^\intercal E^{-1}b$
and by substituting $\xi_i = E^\intercal x_i$ we prove the second part.
The proof is complete. 
\end{proof}

\section{Primal-Dual Newton Conjugate Gradients}\label{sec:pdnewt}
In this section we describe a variation of Newton-CG, which we name primal-dual Newton-CG (pdNCG), for the solution of the primal-dual optimality conditions \eqref{bd33}. The method is similar to the one
in \cite{ctpdnewton} for signal reconstruction problems, although, the two approaches differ in step $3$ of pdNCG. Additionally, we make a step further
and give complete convergence analysis and worst-case iteration complexity results in Section \ref{sec:ncg}. 
A detailed pseudo-code of the method is given below. 
\begin{algorithm}[H]
\begin{algorithmic}[1]
\vspace{0.1cm}
\STATE \textbf{Loop:} For $k=1,2,...$, until $\|d^k\|_{x^k}\le \epsilon$, where $\epsilon >0$. \vspace{0.1cm}
\STATE \hspace{0.83cm} Obtain $d^k$ by solving approximately the system  
\begin{equation}
H(x^k,y^k)d = -\nabla f_\tau^\mu(x^k)
\end{equation}
\hspace{0.83cm} using CG or PCG, where 
\begin{equation}\label{bd35}
H(x,y) = \tau D(I - Ddiag(x)diag(y)) + \nabla^2 \varphi(x)
\end{equation} 
\hspace{0.83cm} and matrix $D$ is defined in \eqref{bd34}. Obtain $\Delta y^k$ by calculating \vspace{0.1cm}
\begin{equation}\label{bd5000}
\Delta y^k = D(I-Ddiag(x)diag(y))d - (y^k - Dx^k).
\end{equation}
\STATE \hspace{0.83cm} Set 
$
\tilde{y}^{k+1}  =  y^k + \Delta y^k
$
and calculate
\begin{equation*}
y^{k+1} :=  P_{\|\cdot\|_\infty\le1}(\tilde{y}^{k+1}),
\end{equation*}
\hspace{0.83cm} where $P_{\|\cdot\|_\infty\le1}(\cdot)$ is the orthogonal projection in the $\ell_\infty$ ball. \\ 
\vspace{0.1cm}
\STATE \hspace{0.83cm} Find the least integer $j\ge0$ such that the function $f_\tau^\mu(x)$ is sufficiently \\ 
             \hspace{0.83cm} decreased along $d^k$
\begin{equation*}
f_\tau^\mu(x^k + c_3^j d^k) \le f_\tau^\mu(x^k) - c_2c_3^j \|d^k\|_{x^k}^2,
\end{equation*}
 \hspace{0.83cm} where $0<c_2<1/2$, $0< c_3<1$,
 and set $\alpha = c_3^j$. \vspace{0.1cm}
\STATE \hspace{0.83cm}  Set $x^{k+1} = x^k + \alpha d^k$. 
\end{algorithmic}
\caption*{Algorithm pdNCG}
\end{algorithm}
In algorithm pdNCG we make use of the local norm
\begin{equation}\label{bd85}
\|\cdot\|_{x^k}:=\sqrt{\langle \cdot, H(x^k,y^k) \cdot \rangle},
\end{equation}
where $H(x^k,y^k)$ is a positive definite matrix under the condition that $\|y^k\|_\infty \le 1$ (Lemma \ref{lem:14}).
Step $2$ of pdNCG is the approximate solution of the linearization of the first two equations in \eqref{bd33}. The matrix $H(x,y)$ is obtained by simply eliminating the variables $\Delta y^k$
in the linearized system. Step $2$ is performed by CG or PCG, which is always initialized with the zero solution and it is terminated when 
\begin{equation}\label{bd32}
\|r_\tau^\mu(x,y)\| \le \eta\|\nabla f_\tau^\mu(x)\|,
\end{equation}
where $r_\tau^\mu(x,y)=H(x,y) d+ \nabla f_\tau^\mu(x)$ is the residual and $0 \le \eta < 1$ is a user-defined constant.
In practice we have observed that setting $\eta^k=$1.0e$\mminus$1
results in very fast convergence, however, the method will be analyzed for $\eta^k$ set as in 
\begin{equation}\label{bd97}
\eta^k=\displaystyle\min\{\frac{1}{2},\|\nabla f_\tau^\mu(x^k)\|^{c_0}\},
\end{equation}
with $c_0=1$.

Step $3$ is a projection of $\tilde{y}^{k+1}$ to the set $\|y\|_\infty \le 1$ such that feasibility of the third condition in \eqref{bd33} is always maintained. The projection 
operator is 
$$
v :=  P_{\|\cdot\|_\infty\le1}(u) = \mbox{sign}(u)\mbox{min}(|u|,1)
$$
and it is applied component-wise. Step $4$ is a backtracking line-search technique in order to guarantee that the sequence $\{x^k\}$ generated by pdNCG monotonically decreases the objective 
function $f_\tau^\mu(x)$.
 
\section{Convergence analysis and worst-case iteration complexity}\label{sec:ncg}
In this section we analyze the pdNCG method. In particular, we prove global convergence, we study the global and local convergence rates and we explicitly define a region in which pdNCG has fast convergence rate. Additionally, worst-case iteration complexity result of pdNCG is presented. The reader will notice that the results in this section are established when CG is used in step $2$ of pdNCG. However, based on Lemma \ref{lem:15} it is trivial 
to show that the same results hold if PCG is used.

Before we introduce notational conventions for this section, it is necessary to find uniform bounds for matrix $H(x,y)$ in \eqref{bd35}. This is shown in the following lemma.
\begin{lemma}\label{lem:14}
If $\|y\|_\infty \le 1$, then matrix $H(x,y)$ is uniformly bounded by
\begin{equation*}
 \lambda_m I \prec H(x,y) \preceq \Big(\frac{\tau}{\mu} + \lambda_1\Big)  I,
\end{equation*}
where $I$ is the identity matrix of appropriate dimension.
\end{lemma}
\begin{proof}
This result easily follows by using the definition of $H(x,y)$ in \eqref{bd35} and \eqref{bd15}. A similar argument, but for signal reconstruction problems, is also claimed in \cite{ctpdnewton}, page $1970$.
The proof is complete.
\end{proof}
The equivalence of the Euclidean and the local norm \eqref{bd85} if $\|y\|_\infty \le 1$, is given by the following inequality
 \begin{equation}\label{bd69}
 \lambda_m^{\frac{1}{2}}\|d\| \le \|d\|_x \le \Big(\frac{\tau}{\mu}+\lambda_1\Big)^{\frac{1}{2}}\|d\|.
 \end{equation}
The upper bound of the largest eigenvalue
of $H(x,y)$ if $\|y\|_\infty \le 1$, will be denoted by $\tilde{\lambda}_1=({\tau}/{\mu} + \lambda_1)$. 
An upper bound of the condition number of matrix $H(x,y)$ will be denoted by $\kappa = {\tilde{\lambda}_1}/{\lambda_m}$.
The Lipschitz constant $L_{f_\tau^\mu}$ defined in Lemma \ref{lem:7},
will be denoted by $L$. Finally, the indices $\tau$ and $\mu$ from function $f_\tau^\mu(x)$ are dropped.  

\subsection{Global convergence}\label{subset:mondec}
First, the minimum decrease of the objective function at every iteration of pdNCG is calculated.
\begin{lemma}\label{lem:monotonic}
Let $x\in\mathbb{R}^m$ be the current iteration of pdNCG, $d\in\mathbb{R}^m$ be the pdNCG direction for the primal variables, which is calculated using CG.
The parameter $\eta$ of the termination criterion \eqref{bd32} of CG is set to $0\le \eta < 1$. 
If $x$ is not the minimizer of problem \eqref{prob2}, i.e., $\nabla f(x)\neq 0$, then the backtracking line-search algorithm in step $4$ of pdNCG will calculate a step-size $\bar{\alpha}$ such that
$$
\bar{\alpha} \ge c_3\frac{\lambda_m}{\tilde{\lambda}_1}.
$$ 
For this step-size $\bar{\alpha}$ the following holds
$$
f(x)-f(x(\bar{\alpha})) > c_4\|d\|_x^2,
$$
where  $c_4=c_2c_3\frac{1}{\kappa}$ and $x(\bar{\alpha})= x + \bar{\alpha} d$. 
\end{lemma}
\begin{proof}
For $x(\alpha)=x + \alpha d$ and from smoothness of $f(x)$ we have 
$$
f(x(\alpha)) \le f(x) + \alpha\nabla f(x)^\intercal d + \frac{\alpha^2}{2}\tilde{\lambda}_1\|d\|^2.
$$
From Lemma \ref{lem:14} we have that $H(x,y)$ is positive definite if $\|y\|_\infty \le 1$, which is the condition that always satisfied by step $3$ of pdNCG.
Then, if $\nabla f(x)\neq 0$ the CG algorithm terminated at the $i^{th}$ iteration returns the vector $d_i\neq 0$, which according to Lemma \ref{lem:15} satisfies
$$
d_i^\intercal H(x,y) d_i = - d_i^\intercal \nabla f(x).
$$
Therefore, by setting $d:= d_i$ we get
$$
f(x(\alpha)) \le f(x) - \alpha\|d\|_x^2 + \frac{\alpha^2}{2}\tilde{\lambda}_1\|d\|^2.
$$
Using \eqref{bd69} we get
$$
f(x(\alpha)) \le f(x) - \alpha\|d\|_x^2 + \frac{\alpha^2}{2}\frac{\tilde{\lambda}_1}{\lambda_m}\|d\|_x^2.
$$
The right hand side of the above inequality is minimized for $\alpha^*=\frac{\lambda_m}{\tilde{\lambda}_1}$, which gives 
$$
f(x(\bar{\alpha})) \le f(x) - \frac{1}{2}\frac{\lambda_m}{\tilde{\lambda}_1}\|d\|_x^2.
$$
Observe that for this step-size the exit condition of the backtracking line-search algorithm is satisfied, since
$$
f(x(\bar{\alpha})) \le f(x) - \frac{1}{2}\frac{\lambda_m}{\tilde{\lambda}_1}\|d\|_x^2 < f(x) - c_2\frac{\lambda_m}{\tilde{\lambda}_1}\|d\|_x^2.
$$
Therefore the step-size $\bar{\alpha}$ returned by the backtracking line-search algorithm is in worst-case bounded by
$$
\bar{\alpha} \ge c_3\frac{\lambda_m}{\tilde{\lambda}_1},
$$ 
which results in the following decrease of the objective function
$$
f(x)-f(x(\bar{\alpha})) > c_2c_3\frac{\lambda_m}{\tilde{\lambda}_1}\|d\|_x^2= c_2c_3\frac{1}{\kappa}\|d\|_x^2.
$$
The proof is complete.
\end{proof}

Global convergence of pdNCG for the primal variables is established in the following theorem.
\begin{theorem}\label{thm:1}
Let $\{x^k\}$ be a sequence generated by pdNCG. 
The parameter $\eta$ of the termination criterion \eqref{bd32} of the CG algorithm is set to $0\le \eta < 1$.
Then the sequence $\{x^k\}$ converges to $x^*$, which is the minimizer of $f(x)$ in problem (\ref{prob2}).
\end{theorem}
\begin{proof}
  From Lemma \ref{lem:14} and step $3$ of pdNCG we have that matrix $H(x,y)$ is symmetric and positive definite at any $x^k,y^k$. Moreover, if $0\le \eta < 1$ in \eqref{bd32}, then 
  CG returns $d^k=0$ at a point $x^k$ if and only if $\nabla f(x^k)=0$. Hence, only at optimality CG will return a zero direction. 
  Moreover, from Lemma \ref{lem:monotonic} we get that if $\nabla f(x^k)\neq 0$, then $\bar{\alpha}^k$ is bounded away from zero and
  the function $f(x)$ is monotonically decreasing when the step $\bar{\alpha}^k d^{k}$ is applied.
  The monotonic decrease of the objective function implies that $\{f(x^k)\}$ converges 
  to a limit, thus, $\{f(x^{k}) - f(x^{k+1})\}\to 0$. Since $f(x^0) < \infty$ and $f(x)$ is monotonically decreased, where $x^0$ is a finite first guess given as an input to pdNCG, 
  then the sequence $\{x^k\}$ belongs in a closed, bounded and therefore, compact sublevel set. Hence, the sequence $\{x^k\}$  must have a subsequence, which converges to a point $x^*$
  and this implies that $\{x^k\}$ also converges to $x^*$.
  Using Lemma \ref{lem:monotonic} and 
  $\{f(x^{k}) - f(x^{k+1})\}\to 0$ we get that $\|d^k\|_x\to 0$, hence, due to positive definiteness of $H(x,y)$, $\|d^k\|\to 0$, which implies that $\|\nabla f(x^k)\|\to 0$. 
  Therefore, $x^*$
  is a stationary point of function $f(x)$. Strong convexity of $f(x)$ guarantees that a stationary point must be a minimizer. 
  The proof is complete.
\end{proof}

Convergence of the dual variables is shown in the following theorem.
\begin{theorem}\label{thm:8}
Let the assumptions of Theorem \ref{thm:1} hold. Then we have that the sequences of dual variables produced by pdNCG satisfy 
$\{y^k\} \to Dx^*$, where $x^*$ is the optimal solution of problem \eqref{prob2}. Furthermore, the previous implies that the primal-dual iterates of
pdNCG converge to the solution of system \eqref{bd33}.
\end{theorem}
\begin{proof}
From Theorem \ref{thm:1} we have that $d^k \to 0$ and $x^k \to x^*$. Hence, from \eqref{bd5000} we get that $\Delta y^k \to -y^k + Dx^*$. 
Moreover, we have that 
the iterates at step $3$ of pdNCG $\tilde{y}^k  \to  Dx^*$ and consequently
\begin{align*}
y^{k}  & =  P_{\|\cdot\|_\infty\le1}(\tilde{y})  \to P_{\|\cdot\|_\infty\le1}(Dx^*) = Dx^*.
\end{align*}
It is easy to check that these values of $y^*$ with the optimal variable $x^*$ satisfy the system \eqref{bd33}. The proof
is complete.
\end{proof}

\subsection{Region of fast convergence rate}\label{sec:region}
In this subsection we define a region based on $\|d\|_x$, in which by setting parameter $\eta$ as in \eqref{bd97} with $c_0=1$, pdNCG converges with fast rate. 
The lemma below shows the behaviour of the function $f(x)$ when a step along the primal pdNCG direction is made.
\begin{lemma}\label{lem:1}
Let $x\in\mathbb{R}^m$ be the current iteration of pdNCG, $d\in\mathbb{R}^m$ be the pdNCG direction for primal variables calculated by CG, which is terminated according to
criterion (\ref{bd32}) with $0 \le \eta < 1$.
Then
\begin{equation*}
f(x)-f(x(\alpha)) \ge \alpha\|d\|_x^2 - \frac{\alpha^2}{2}\|d\|_x^2 - \frac{\alpha^3}{6}\frac{L}{{\lambda}_m^{\frac{3}{2}}} \|d\|_x^3,
\end{equation*}
where $x(\alpha)= x + \alpha d$ and $\alpha>0$.
\end{lemma}
\begin{proof}
Using Lemma \ref{lem:2} and setting $y=x(\alpha)=x+\alpha d$ we get
$$
f(x(\alpha))\le f(x) + \alpha\nabla f(x)^\intercal d + \frac{\alpha^2}{2}d^\intercal \nabla^2 f(x) d + \frac{\alpha^3}{6}L\|d\|^3.
$$
From Lemma \ref{lem:14} and step $3$ of pdNCG we have that \eqref{bd69} holds. Hence, using \eqref{bd69} and Lemma \ref{lem:15} we get
$$
f(x(\alpha))\le f(x) - \alpha\|d\|_x^2 + \frac{\alpha^2}{2}\|d\|_x^2 + \frac{\alpha^3}{6}\frac{L}{\lambda_m^{\frac{3}{2}}}\|d\|_x^3.
$$
The result is obtained by rearrangement of terms.
The proof is complete.
\end{proof}

The next  lemma determines bounds on the norm of the 
primal direction $d^k$ as a function of $\|\nabla f_\tau^\mu(x^k)\|$.
\begin{lemma}\label{lem:CG}
Let $d\in\mathbb{R}^m$ be the pdNCG primal direction calculated by CG, which is terminated according to
criterion (\ref{bd32}) with $0 \le \eta < 1$. Then the following holds
\begin{equation*}
\frac{1-\eta^2}{2\tilde{\lambda}_1^{\frac{1}{2}}}\|\nabla f(x)\| \le \|d\|_x \le \frac{1}{\lambda_m^{\frac{1}{2}}}\|\nabla f(x)\| 
\end{equation*}
\end{lemma}
\begin{proof}
By squaring \eqref{bd32} and making simple rearrangements of it we get
\begin{equation}\label{bd86}
d^\intercal H(x,y)^2d + 2\nabla f(x)^\intercal H(x,y)d + (1-\eta^2)\|\nabla f(x)\|^2 \le 0.
\end{equation}
From step $3$ of pdNCG we have that the condition of Lemma \ref{lem:14} is satisfied. Therefore by using Lemma \ref{lem:14} and Cauchy-Schwarz inequality in \eqref{bd86} we get
\begin{equation*}
\lambda_m^2\|d\|^2 - 2\tilde{\lambda}_1^{\frac{1}{2}} \|\nabla f(x)\|\|d\|_x + (1-\eta^2)\|\nabla f(x)\|^2 \le 0.
\end{equation*}
By dropping the quadratic term $\lambda_m^2\|d\|^2$ from the previous inequality and dividing by $\|\nabla f(x)\|$, after making appropriate rearrangements we get
\begin{equation*}
\|d\|_x \ge \frac{1-\eta^2}{2\tilde{\lambda}_1^{\frac{1}{2}}}\|\nabla f(x)\|.
\end{equation*}
This proves the left hand side of the result. For the right hand side, we simply use Lemma \ref{lem:15} and \eqref{bd69}
$$
d^\intercal H(x,y)d=\|d\|_x^2=-d^\intercal \nabla f(x)\le \|d\|\|\nabla f(x)\| \le \frac{1}{\lambda_m^{\frac{1}{2}}} \|d\|_x\|\nabla f(x)\|.
$$
By dividing with $\|d\|_x$ we obtain the right hand side of our claim.
The proof is complete.
\end{proof}

The following lemma will be used to prove local fast convergence rate of pdNCG for the primal variables. 
\begin{lemma}\label{lem:17}
Let the iterates $x^k$ and $y^k$ be produced by pdNCG, then the following holds
\begin{equation*}
\|\nabla^2f(x^k) - H(x^k,y^k)\| \le \gamma \|d^k\|_{x^k},
\end{equation*}
where 
$$
\gamma = \Big(\frac{8\tilde{\lambda}_1^{\frac{1}{2}}}{\lambda_m}(L + M + \frac{M}{\mu}) + \frac{M}{\lambda_m^{\frac{1}{2}}\mu}\Big),
$$
$M$ is a positive constant.
\end{lemma}
\begin{proof}
Let $x^*$ be the optimal solution of problem \eqref{prob2}. We rewrite
$$
\nabla^2 f(x^k) - H(x^k,y^k) = \nabla^2 f(x^k) - \nabla^2 f(x^*) + \nabla^2 f(x^*) - H(x^k,y^k).
$$
Moreover, let $y^*$ be the optimal dual variable, which according to Theorem \ref{thm:8} satisfies $y^*=D(x^*)x^*$.
Notice that matrix $D$ in \eqref{bd34} is dependent on variable $x$; for the purposes of this proof we will explicitly denote this dependence. 
From the definition of $H(x,y)$ in \eqref{bd35} we have that $H(x^*,y^*)=\nabla^2 f(x^*)$. The following holds
$$
\|\nabla^2 f(x^k) - H(x^k,y^k) \| \le \|\nabla^2 f(x^k) - \nabla^2 f(x^*)\| + \|H(x^*,y^*) - H(x^k,y^k)\|.
$$
By Lipschitz continuity of $\nabla^2 f(x)$ in Lemma \ref{lem:7} we get that 
\begin{equation}\label{bd87}
\|\nabla^2 f(x^k) - H(x^k,y^k) \| \le L\| x^* - x^k\| + \|H(x^*,y^*) - H(x^k,y^k)\|.
\end{equation}
We now focus on bounding $\|H(x^*,y^*) - H(x^k,y^k)\|$. Using the fundamental theorem of calculus  we have
$$
H(x^*,y^*) - H(x^k,y^k) = \int_0^1\frac{d H(x^*(s),y^*(s))}{d(x^*(s),y^*(s))}[x^*-x^k; y^* - y^k]ds,
$$
where $x^*(s)= x^* + s(x^*- x^k)$ and $y^*(s) = y^* + s(y^*- y^k)$. Hence,
$$
\|H(x^*,y^*) - H(x^k,y^k)\| \le (\|x^* - x^k\| + \|y^* - y^k\|) \int_0^1\Big\| \frac{d H(x^*(s),y^*(s))}{d(x^*(s),y^*(s))} \Big\| ds.
$$
We now prove that ${d H(x^*(s),y^*(s))}/{d(x^*(s),y^*(s))}$ is bounded in the set $\mathbb{R}^m\times\{y\in\mathbb{R}^m \ | \ \|y\|_\infty \le 1\} \subset \mathbb{R}^{2m}$.
Observe that the partial derivatives $H(x,y)$ with respect to $x$ or $y$ are continuous. Therefore, ${d H(x^*(s),y^*(s))}/{d(x^*(s),y^*(s))}$ is a continuous tensor.
In this case, the only candidates of unboundedness are the limits $x\to \pm \infty$. It is easy to show that at the limits all partial derivatives are finite and this implies 
that every component of the tensor is bounded in $\mathbb{R}^m\times\{y\in\mathbb{R}^m \ | \ \|y\|_\infty \le 1\}$. We will denote the bound by a positive constant $M$, hence,
\begin{equation}\label{bd91}
\|H(x^*,y^*) - H(x^k,y^k)\| \le M(\|x^* - x^k\| + \|y^* - y^k\|).
\end{equation}
It remains to find a bound for $\|y^* - y^k\|$. From step $3$ of pdNCG we have 
\begin{align*}
\|y^* - y^k\| & \le \|P_{\|\cdot\|_\infty\le1}(y^*) - P_{\|\cdot\|_\infty\le1}(\tilde{y}^{k})\| \le \|y^* - \tilde{y}^k\| \\ 
                  & \le \|D(x^*)x^* - D(x^k)x^k\|  \\ 
                  & + \|D(x^k)(I - D(x^k)diag(x^k)diag(y^k))\|\|d^k\| \\
                  & = \|\nabla \psi_\mu(x^*) + \nabla \psi_\mu(x^k)\| \\ 
                  & + \|D(x^k)(I - D(x^k)diag(x^k)diag(y^k))\|\|d^k\|.
\end{align*}
Using Lemma \ref{lem:16} and 
$$D(x^k)(I - D(x^k)diag(x^k)diag(y^k)) \preceq D(x^k) \preceq \frac{1}{\mu} I,$$
which holds for $\|y\|_\infty \le 1$, we have that 
\begin{equation}\label{bd89}
\|y^* - y^k\| \le \frac{1}{\mu}\|x^*-x^k\| + \frac{1}{\mu}\|d^k\|.
\end{equation}
By combining inequalities \eqref{bd91} and \eqref{bd89} in \eqref{bd87} we get
$$
\|\nabla^2 f(x^k) - H(x^k,y^k) \| \le (L + M + \frac{M}{\mu})\|x^* - x^k\| + \frac{M}{\mu}\|d^k\|.
$$
Combining Lemmas \ref{lem:13}, \ref{lem:CG} and \eqref{bd97} for a bound on $\|x^* - x^k\|$ we get 
$$
\|\nabla^2 f(x^k) - H(x^k,y^k) \| \le \frac{8\tilde{\lambda}_1^{\frac{1}{2}}}{\lambda_m}(L + M + \frac{M}{\mu})\|d^k\|_{x^k} + \frac{M}{\mu}\|d^k\|.
$$
Using \eqref{bd69} we get the result. The proof is complete.
\end{proof}

Based on Lemmas \ref{lem:1}, \ref{lem:CG} and \ref{lem:17}, a region is defined in the following lemma, in which unit-step sizes are calculated by the backtracking line-search algorithm. Additionally, for this region,
$\|d^{k+1}\|_{x^{k+1}}$ is bounded as a function of $\|d^{k}\|_{x^{k}}$. In this lemma the constants $c_2$ and $c_3$ have been defined in step $4$ of pdNCG, moreover,
$x^{k+1}=x^k+d^k$. 
\begin{lemma}\label{lem:12}
If 
$
\|d^{k}\|_{x^k}\le 3(1-2c_2)\frac{\lambda_m^{\frac{3}{2}}}{L},
$
then the backtracking line-search algorithm in step $4$ of pdNCG calculates unit step-sizes. Moreover,
if the parameter $\eta^k$ of the termination criterion (\ref{bd32}) of CG is set as in \eqref{bd97} with $c_0=1$, 
then for two consequent primal directions $d^{k}$, $d^{k+1}$ and points $x^k$, $x^{k+1}$, the following holds
\begin{equation*}
\frac{1}{2}\frac{16\tilde{\lambda}_1\lambda_m+2\gamma\lambda_m^{\frac{1}{2}} + L}{\lambda_m^{\frac{3}{2}}}\|d^{k+1}\|_{x^{k+1}} \le \Big(\frac{1}{2}\frac{16\tilde{\lambda}_1\lambda_m+2\gamma\lambda_m^{\frac{1}{2}}+L}{\lambda_m^{\frac{3}{2}}}\|d^k\|_{x^k}\Big)^2.
\end{equation*}
\end{lemma}
\begin{proof}
By setting $\bar{\alpha}=1$ in Lemma \ref{lem:1} we get
$$
f(x^k)-f(x^{k+1}) \ge \frac{1}{2}\|d^k\|_{x^k}^2 - \frac{1}{6}\frac{L}{{\lambda}_m^{\frac{3}{2}}} \|d^k\|_{x^k}^3 = \Big(\frac{1}{2} - \frac{1}{6}\frac{L}{{\lambda}_m^{\frac{3}{2}}}\|d^k\|_{x^k}\Big)\|d^k\|_{x^k}^2.
$$ 
if $\|d^{k}\|_{x^k}\le 3(1-2c_2)\frac{\lambda_m^{\frac{3}{2}}}{L}$ we get
$$
f(x^k)-f(x^{k+1}) \ge c_2\|d^k\|_{x^k}^2,
$$
which implies that $\bar{\alpha}=1$ satisfies the exit condition of the backtracking line-search algorithm.
Let us define the quantities $\nabla f(x(t))^\intercal h$, where $h\in\mathbb{R}^m$, $x(t)=x^k+td^k$ and $x(\delta)=x^k+\delta d^k$ then we have 
\begin{align}
   \nabla f(x(t))^\intercal h & = \nabla f(x^k)^\intercal h  + t (d^k)^\intercal \nabla^2 f(x^k) h \nonumber \\
       				        & + \int_0^t\int_0^u \nabla^3f(x(\delta))[d^k,d^k,h] d\delta du \nonumber \\
                                             &\le \nabla f(x^k)^\intercal h  + t (d^k)^\intercal \nabla^2 f(x^k) h \nonumber \\
                                             &+ \int_0^t\int_0^u \Big|\nabla^3f(x(\delta))[d^k,d^k,h]\Big|d\delta du \nonumber\\
                                             &=   \nabla f(x^k)^\intercal h  + t (d^k)^\intercal \nabla^2 f(x^k) h \nonumber\\
                                             &+ \int_0^t\int_0^u  \lim_{\delta\to 0} \Big|\frac{(d^k)^\intercal (\nabla^2 f(x(\delta)) - \nabla^2 f(x^k)) h}{\delta}\Big| d\delta du\nonumber \\
    				        &\le \nabla f(x^k)^\intercal h  + t (d^k)^\intercal \nabla^2 f(x^k) h \nonumber\\
                                             &+ \|d^k\|\|h\|\int_0^t\int_0^u  \lim_{\delta\to 0}\Big\|\frac{1}{\delta}(\nabla^2 f(x(\delta)) - \nabla^2 f(x^k))\Big\| d\delta du \nonumber\\
     				        &\le \nabla f(x^k)^\intercal h  + t (d^k)^\intercal \nabla^2 f(x^k) h + \|d^k\|\|h\|\int_0^t\int_0^u L\|d^k\| d\delta du \nonumber\\
                                             &= \nabla f(x^k)^\intercal h  + t (d^k)^\intercal \nabla^2 f(x^k) h + \frac{t^2}{2}L \|d^k\|^2\|h\|. \nonumber
\end{align}
By taking absolute values and setting $t=1$ we get
\begin{align}\label{bd92}
|\nabla f(x^{k+1})^\intercal h| & \le |\nabla f(x^{k})^\intercal h  + (d^k)^\intercal \nabla^2 f(x^k) h| + \frac{1}{2}L \|d^k\|^2\|h\| \nonumber \\
                                              & \le \|\nabla f(x^{k}) +\nabla^2 f(x^k) d^k \|\|h\| + \frac{1}{2}L \|d^k\|^2\|h\| \nonumber\\
                                              & \le \|\nabla f(x^{k}) + H(x^k,y^k) d^k\|\|h\| \nonumber\\
                                              & +  \|\nabla^2 f(x^k) - H(x^k,y^k)\|\|d^k\| \| h\| + \frac{1}{2}L \|d^k\|^2\|h\| 
\end{align}
Observe that from \eqref{bd97} with $c^0=1$ we have that $\eta^k\le \|\nabla f(x^k)\|$. Hence, combining the previous
with Lemma \ref{lem:CG} and \eqref{bd32} in \eqref{bd92} we have that 
\begin{align*}
|\nabla f(x^{k+1})^\intercal h| & \le 8\tilde{\lambda}_1\|d^k\|_{x^k}^2\|h\|+  \|\nabla^2 f(x^k) - H(x^k,y^k)\|\|d^k\| \| h\|  \\
                                              & + \frac{1}{2}L \|d^k\|^2\|h\|.
\end{align*}
Using Lemma \ref{lem:17} we have
\begin{align*}
|\nabla f(x^{k+1})^\intercal h| \le 8\tilde{\lambda}_1\|d^k\|_{x^k}^2\|h\|+  \gamma \|d^k\|_{x^k}\|d^k\|\|h\|  + \frac{1}{2}L \|d^k\|^2\|h\|.
\end{align*}
From the equivalence of norms \eqref{bd69} we get 
$$
|\nabla f(x^{k+1})^\intercal h| \le \frac{1}{2}\frac{16\tilde{\lambda}_1\lambda_m+  2\gamma\lambda_m^{\frac{1}{2}}  + L}{\lambda_m^{\frac{3}{2}}} \|d^k\|_{x^k}^2\|h\|_{x^{k+1}}.
$$
The previous result holds for every $h\in\mathbb{R}^m$, hence, by setting $h=d^{k+1}$ and by using Lemma \ref{lem:15} we prove the second part of this lemma. The proof is complete.
\end{proof} 
The following corollary states the region of fast convergence rate of Newton-CG. By fast rate it is meant that if pdNCG is initialized in this region, then the worst-case iteration complexity result for convergence to $x^*$ is of the form $\log_2 \log_2 \frac{\text{constant}}{\text{required accuracy}}$. This statement is proved in Subsection \ref{sub4.4} in Theorem \ref{thm:5}.
\begin{corollary}\label{cor:2}
If the parameter $\eta^k$ in the termination criterion (\ref{bd32}) of CG is set as in \eqref{bd97} with $c_0=1$ and $\|d^k\|_{x^k}< \varpi$, $0< \varpi \le c_5$, where
$$
c_5 = \displaystyle\min\Big\{3(1-2c_2)\frac{{\lambda}_m^{\frac{3}{2}}}{L}, \frac{\lambda_m^{\frac{3}{2}}}{16\tilde{\lambda}_1\lambda_m+2\gamma\lambda_m^{\frac{1}{2}} + L}  \Big\},
$$
then according to Lemma \ref{lem:12} pdNCG convergences with fast rate.
\end{corollary}

\subsection{Worst-case iteration complexity}\label{sub4.4}
The following theorem shows the worst-case iteration complexity of pdNCG in order to enter the region of fast convergence rate, i.e., $\|d\|_x < \varpi$, where $0<\varpi\le c_5$ and $c_5$
has been defined in Corollary \ref{cor:2}. In this theorem the constant $c_4$ has been defined in Lemma \ref{lem:monotonic}, $c_2$ and $c_3$ are constants of the backtracking line-search algorithm in step $4$ of pdNCG. 
Moreover, $x^*$ denotes the minimizer of problem \eqref{prob2}.
\begin{theorem}\label{thm:4}
Starting from an initial point $x^0$, such that $\|d^0\|_{x^0} \ge \varpi$ and setting $0\le \eta < 1$ in the termination criterion \eqref{bd32} of CG, then pdNCG requires at most 
$$
K_1 = c_6\log\Big(\frac{f(x^0)-f(x^*)}{c_7\varpi^2}\Big),
$$
iterations to obtain a solution $x^k$, $k>0$, such that $\|d^k\|_{x^k} < \varpi$, where 
$$
c_6=\frac{2\kappa^2}{(1-\eta^2)^2c_2c_3} \quad \mbox{and} \quad c_7=\frac{1}{2\kappa}.
$$
\end{theorem}
\begin{proof}
Let us assume an iteration index $k>0$, then from Lemmas \ref{lem:13} and \ref{lem:CG} we get
\begin{equation}\label{bd99}
f(x^{k})-f(x^*) \ge \frac{1}{2\kappa}\|d^{k}\|_{x^{k}}^2,
\end{equation}
and 
\begin{equation}\label{bd100}
f(x^{k-1})-f(x^*) \le \frac{2\kappa}{(1-\eta^2)^2}\|d^{k-1}\|_{x^{k-1}}^2.
\end{equation}
From Lemma \ref{lem:monotonic} we have
\begin{equation}\label{bd101}
f(x^{k})< f(x^{k-1}) - c_4\|d^{k-1}\|_{x^{k-1}}^2.
\end{equation}
Combining \eqref{bd100}, \eqref{bd101} and subtracting $f(x^*)$ from both sides we get
\begin{align}
f(x^{k})-f(x^*) & < \Big(1-\frac{(1-\eta^2)^2c_4}{2\kappa}\Big)(f(x^{k-1})-f(x^*)) \nonumber \\
												     & < \Big(1-\frac{(1-\eta^2)^2c_4}{2\kappa}\Big)^k(f(x^0)-f(x^*)) \nonumber \\
												     &= \Big(1-\frac{(1-\eta^2)^2c_2c_3}{2\kappa^2}\Big)^k(f(x^0)-f(x^*))  \nonumber
\end{align}
From the last inequality and \eqref{bd99} we get
$$
\frac{1}{2\kappa}\|d^{k}\|_{x^{k}}^2 < \Big(1-\frac{(1-\eta^2)^2c_2c_3}{2\kappa^2}\Big)^k(f(x^0)-f(x^*)).
$$
Using the definitions of constants $c_6$ and $c_7$ we have
$$
\|d^{k}\|_{x^{k}}^2 < \Big(1-\frac{1}{c_6}\Big)^k\frac{1}{c_7}(f(x^0)-f(x^*)).
$$
Hence, we conclude that after at most $K_1$ iterations as defined in the preamble of this theorem, the algorithm produces $\|d^{k}\|_{x^{k}} < \varpi$.
The proof is complete.
\end{proof}

It is worth pointing out that a worst-case iteration complexity result for the global phase (before fast local convergence) of standard Newton
method can be obtained by slightly modifying the proof of Theorem \ref{thm:4}.
In particular, in proof of Theorem \ref{thm:4}, one simply has to replace matrix $H$ with $\nabla^2 f(x)$ and
set $\eta=0$ (exact Newton directions), the constant $\kappa$ remains unchanged since the matrices $H$ and 
$\nabla^2 f(x)$ have the same uniform bounds, see \eqref{bd2} and Lemma \ref{lem:14}.
Using the previous adjustments, it is easy to show that the dominant term $\kappa^2$ in the result of Theorem \ref{thm:4} is preserved
for standard Newton method. To the best of our knowledge, the result of $\mathcal{O}(\kappa^2)$ is the tightest that has been
 obtained for standard Newton method, see Subsection $9.5$ in \cite{bookboyd}.

The following theorem presents the worst-case iteration complexity result of pdNCG to obtain a solution $x^l$, of accuracy $f(x^l)-f(x^*)<\epsilon$, when
initialized at a point inside the region of fast convergence. 
\begin{theorem}\label{thm:5}
Suppose that there is an iteration index $k$ of pdNCG, such that $\|d^k\|_{x^k} < \varpi$. If $\eta$ in \eqref{bd32} is set as in \eqref{bd97} with $c_0=1$, then pdNCG needs at most
$$
K_2 = \log_2\log_2\Big(\frac{c_8}{\epsilon}\Big)
$$
additional iterations to obtain a solution $x^l$, $l>k$, such that
$
f(x^l)-f(x^*) < \epsilon,
$
where 
$$
c_8 = \frac{16\kappa\lambda_m^3}{(16\tilde{\lambda}_1\lambda_m+2\gamma\lambda_m^{\frac{1}{2}} + L)^2} .
$$
\end{theorem}
\begin{proof}
Suppose that there is an iteration index $k$ such that $\|d^k\|_{x^k}<\varpi$, then for an index $l>k$, by applying Lemma \ref{lem:12} recursively we get
\begin{align}\label{bd102}
\frac{1}{2}\frac{16\tilde{\lambda}_1\lambda_m+2\gamma\lambda_m^{\frac{1}{2}} + L}{\lambda_m^{\frac{3}{2}}}\|d^{l}\|_{x^{l}} & \le \Big(\frac{1}{2}\frac{16\tilde{\lambda}_1\lambda_m+2\gamma\lambda_m^{\frac{1}{2}} + L}{\lambda_m^{\frac{3}{2}}}\|d^k\|_{x^k}\Big)^{2^{l-k}}  \nonumber \\
             &< \Big(\frac{1}{2}\Big)^{2^{l-k}}.
\end{align}
From Lemmas \ref{lem:13}, \ref{lem:CG} and $\eta^k$ in \eqref{bd97} we get
\begin{equation*}
f(x^l)-f(x^*) \le {4\kappa}\|d^l\|_{x^l}^2,
\end{equation*}
By replacing \eqref{bd102} in the above inequality we get
$$
f(x^l)-f(x^*) < \frac{16\kappa\lambda_m^3}{(16\tilde{\lambda}_1\lambda_m+2\gamma\lambda_m^{\frac{1}{2}} + L)^2} \Big(\frac{1}{2}\Big)^{2^{l-k+1}}.
$$
Hence, in order to obtain a solution $x^l$, such that $f(x^l)-f(x^*)< \epsilon$, pdNCG requires at most as many iterations as in the preamble of this theorem.
The proof is complete.
\end{proof}

The following theorem summarizes the complexity result of pdNCG. The constants $c_6$, $c_7$ and $c_8$
in this theorem are defined in Theorems \ref{thm:4} and \ref{thm:5}, respectively. 
\begin{theorem}\label{thm:7}
Starting from an initial point $x^0$, such that $\|d^0\|_{x^0} \ge \varpi$, pdNCG requires at most
$$
K_3 = c_6\log\Big(\frac{f(x^0)-f(x^*)}{c_7\varpi^2}\Big) + \log_2\log_2\Big(\frac{c_8}{\epsilon}\Big)
$$
iterations to converge to a solution $x^k$, $k>0$, of accuracy 
$$
f(x^k)-f(x^*)< \epsilon.
$$
\end{theorem}

\section{Numerical Experiments}\label{sec:nexp}
We illustrate the robustness and efficiency of pdNCG on synthetic $\ell_1$-regularized Sparse Least-Squares (S-LS) problems and real world $\ell_1$-regularized Logistic Regression (LR) problems. 
 
 \subsection{State-of-the-art first-order methods}
A number of efficient first-order methods \cite{fista,changHsiehLin,HsiehChang,petermartin,peterbigdata,shwartzTewari,tsengblkcoo,nesterovhuge,tsengyun,wrightaccel,tonglange} have been developed for the solution of problem (\ref{prob1}). 
The most efficient first-order methods, for example \cite{fista,peterbigdata}, rely on properties of the $\ell_1$-norm to obtain the new direction at each iteration. 
In particular, at every iteration they require the exact minimization of a subproblem 
\begin{equation}\label{bd81}
\min_{x^+} \ \ \tau \|x^+\|_1 + \varphi(x) + \nabla \varphi(x)^\intercal (x^+-x) + \frac{L_\varphi}{2} \|x^+-x\|^2,
\end{equation}
where $x$ is a given point. Other first-order methods use the decomposability of the former problem and solve it only for some chosen coordinates \cite{peterbigdata}. 
In this case, the Lipschitz constant is replaced by partial Lipschitz constants for each chosen coordinate.

 
In this section we compare pdNCG with two such state-of-the-art first-order methods.

\begin{itemize}
\item FISTA (Fast Iterative Shrinkage-Thresholding Algorithm) \cite{fista} is an optimal first-order method for problem \eqref{prob1}.  
An efficient implementation of this algorithm can be found as part 
of the TFOCS (Templates for First-Order Conic Solvers) package \cite{convexTemplates} under the name N$83$.

\item PCDM (Parallel Coordinate Descent Method) \cite{peterbigdata}. The published implementation performs parallel coordinate updates asynchronously based on \eqref{bd81}, where the coordinates are chosen uniformly at random.
This method is well-known for exploiting separability of the problems.
\end{itemize}

\subsection{Implementation details}
Solver PCDM is a $C$++ implementation, while FISTA and pdNCG are implemented in MATLAB.  
We expect that the programming language will not be an obstacle for FISTA and pdNCG. This is because these methods rely only on basic linear algebra operations, such as the dot product, which are implemented in $C$++ 
in MATLAB by default. 
All experiments are performed on a Dell PowerEdge R920 running Redhat Enterprise Linux with four Intel Xeon E7-4830 v2 2.2GHz processors, 20M Cache, 7.2 GT/s QPI, Turbo (4x10Cores). 
PCDM as a parallel method exploits $40$ cores. Whilst, FISTA and pdNCG, which are MATLAB implementations, exploit multicore systems by performing in parallel simple linear algebra tasks by default.
Finally, for pdNCG a simple diagonal preconditioner is used for all experiments. The preconditioner is set to be the inverse of the diagonal
of matrix $H$.

\subsection{Parameter tuning}


For pdNCG, the smoothing parameter $\mu$ is set to $1.0e$-$4$ and the parameter $\eta$ in \eqref{bd97} is set to $1.0e$-$1$. The backtracking step-size
of pdNCG is set to $c_3=1/2$ and the parameter of sufficient decrease is set to $c_2=1.0e$-$3$.
For PCDM, the parameter $\sigma$ is set to $1 + 23(\nu-1)/(m-1)$ like it is proposed in \cite{peterbigdata}, where $\nu$ is the partial separability degree 
of the problem that is solved; we will define $\nu$ later in this section. 
Moreover, for PCDM making functions evaluations is prohibited, because it is considered as a very expensive operation, hence,
we do not include the running time of making such operations in the total running time. 
For FISTA we use the default parameter setting. All solvers are initialized to the zero solution. 

We run pdNCG for sufficient time such that the problems are adequately solved. Then, FISTA and PCDM are terminated when the objective function $ f_{\tau}(x)$ in \eqref{prob1} 
is below the one obtained by pdNCG or when a predefined maximum number of iterations limit is reached. All comparisons are presented in figures that 
show the progress of the objective function against the wall clock time. This way, the reader can compare the performance of the solvers for various levels of 
accuracy. 

\subsection{$\ell_1$-Regularized Sparse Least-Squares}\label{subset:slsq}
In this subsection we compare pdNCG with FISTA and PCDM. The comparison is made on a problem for which
\begin{equation*}
\varphi(x)=\frac{1}{2}\|Ax-b\|^2
\end{equation*} 
in (\ref{prob1}),
where $x\in\mathbb{R}^m$, $b\in\mathbb{R}^n$, $A\in\mathbb{R}^{n\times m}$ with $n \ge m$. We are interested in problems of this form 
that are sparse, with ill-conditioned $A^\intercal A$ and \textit{partially} or \textit{highly} separable. The definition of separability that is employed 
is the same as in \cite{peterbigdata}, which for these problems is measured with the following constant
\begin{equation}\label{bd83}
\beta_{LS} := \max_{j\in[1,2,\ldots,n]} \|A_j\|_0,
\end{equation}
where $A_j$ is the $j^{th}$ row of matrix $A$. Obviously, the following holds $1\le \beta_{LS} \le m$. Notice that the larger $\beta_{LS}$ is the less separable the problem becomes. 
However, observe that $\beta_{LS}$ captures separability based only on the most dense row of matrix $A$. This implies that there might exist a matrix $A$ that is very sparse but
there is a single row of $A$ that is relatively dense and this will result in large $\beta_{LS}$.
In the examples that will be presented in this subsection $\beta_{LS}$ is a small fraction of $m$.

\subsubsection{Benchmark Generator}


A generator for non-trivial sparse S-LS problems is given in the following simple process. First, a full-rank matrix $A\in\mathbb{R}^{n\times m}$ with $n \ge m$ is generated. Second, the eigenvalue decomposition of $A^\intercal A=Q\Lambda Q^\intercal$ is computed.
Third, the optimal solution is generated by approximately solving
\begin{equation}\label{eq:130}
\begin{array}{cll}
x^*:= & \displaystyle\argmin_{x\in\mathbb{R}^{n}} & \|Q^\intercal x - \Lambda^{-1} e\|^2 \\
& \mbox{subject to:} & \|x\|_0 \le s, \\
\end{array}
\end{equation}
where $e$ is a vector of ones, $\|\cdot\|_0$ is the zero norm, which counts the number of nonzero components of the input argument and $s$ is a positive integer.
To solve the above problem one can use an Orthogonal Matching Pursuit (OMP) \cite{cosamp} 
solver implemented in \cite{cosampImpl}.
The aim of this approach is to find a sparse $x^*$, which can be expressed as $x^*=Qv$, where the coefficients $v$ of the linear combination are close to the inverse of the eigenvalues of matrix $A^\intercal A$. 
Intuitively, this technique will create an $x^*$, which has strong dependence on subspaces that correspond to the smaller eigenvalues of $A^\intercal A$.
It is well known that such an $x^*$ makes the problem difficult to solve, see for example the analysis of Steepest Descent for LS in \cite{Shewchuk94anintroduction}. Finally, $b$ can be generated such that
the following holds
$$
x^*:=  \displaystyle\argmin_{x\in\mathbb{R}^{n}}  \|x\|_1 + \|Ax-b\|^2.
$$ 
This is achieved by substituting in the optimality conditions of the previous problem $x^*$ and then choosing $b$ such that the optimality conditions are satisfied.
It is easy to extend the generator and consider a minimization of $\min_{x\in\mathbb{R}^{n}}  \tau \|x\|_1 + \|Ax-b\|^2$ instead.


\subsubsection{Synthetic sparse least-squares example: increasing conditioning}
We now present the performance of pdNCG, FISTA and PCDM for increasing condition number of matrix $A^TA$. 
We generate six instances $(A,b,x^*)$, where the condition number of $A^\intercal A$ takes values $1.00e$\scalebox{.75}{$\mplus$}$02$, $1.00e$\scalebox{.75}{$\mplus$}$04$,
$1.00e$\scalebox{.75}{$\mplus$}$06$, $1.00e$\scalebox{.75}{$\mplus$}$08$, $1.00e$\scalebox{.75}{$\mplus$}$10$ and $1.00e$\scalebox{.75}{$\mplus$}$12$.
Matrix $A$ has $m=2^{22}$ columns, $n= 2 m$ rows and rank $m$. 
Moreover, matrix $A$ is sparse, i.e., $nnz(A)/(mn)\approx 3.00e$-$07$ and $\beta_{LS}= 2$.
The optimal solution $x^*$ has approximately $s \approx 8.0e$-$03 m$ non-zero components.

The results of this experiment are shown in Figure \ref{fig2}. 
In this figure the objective function $f_\tau(x)$ is presented against the wall clock time for each solver. Observe the log-scale used for both axes. 
The wall clock time of the solvers is shown after their first iteration takes place. 
PCDM was the fastest method for condition number less than or equal to $1.00e$\scalebox{.75}{$\mplus$}$04$, while pdNCG was the second fastest method. 
For condition number $1.00e$\scalebox{.75}{$\mplus$}$06$ pdNCG converged in comparable time with PCDM, which was the fastest. 
For condition number larger than or equal to $1.00e$\scalebox{.75}{$\mplus$}$08$ pdNCG was clearly the fastest method. Moreover, for condition
number larger than or equal to $1.00e$\scalebox{.75}{$\mplus$}$10$ pdNCG was the only method that solved the problems to sufficient accuracy
within reasonable time. 
Notice that despite the problems being sparse and highly separable, PCDM and FISTA did not scale well for the problems with condition number larger than or equal to $1.0e$+$10$. 
In particular, when the condition number of $A^\intercal A$ is $1.00e$\scalebox{.75}{$\mplus$}$12$, PCDM and FISTA did not converge in competitive time; they were terminated after more than $27$ hours of wall clock time. 
\begin{figure}%
  \centering

  \subfloat[$\kappa(A^\intercal A) = 1.0e$+$02$]{\label{fig2a}\includegraphics[scale=0.37]{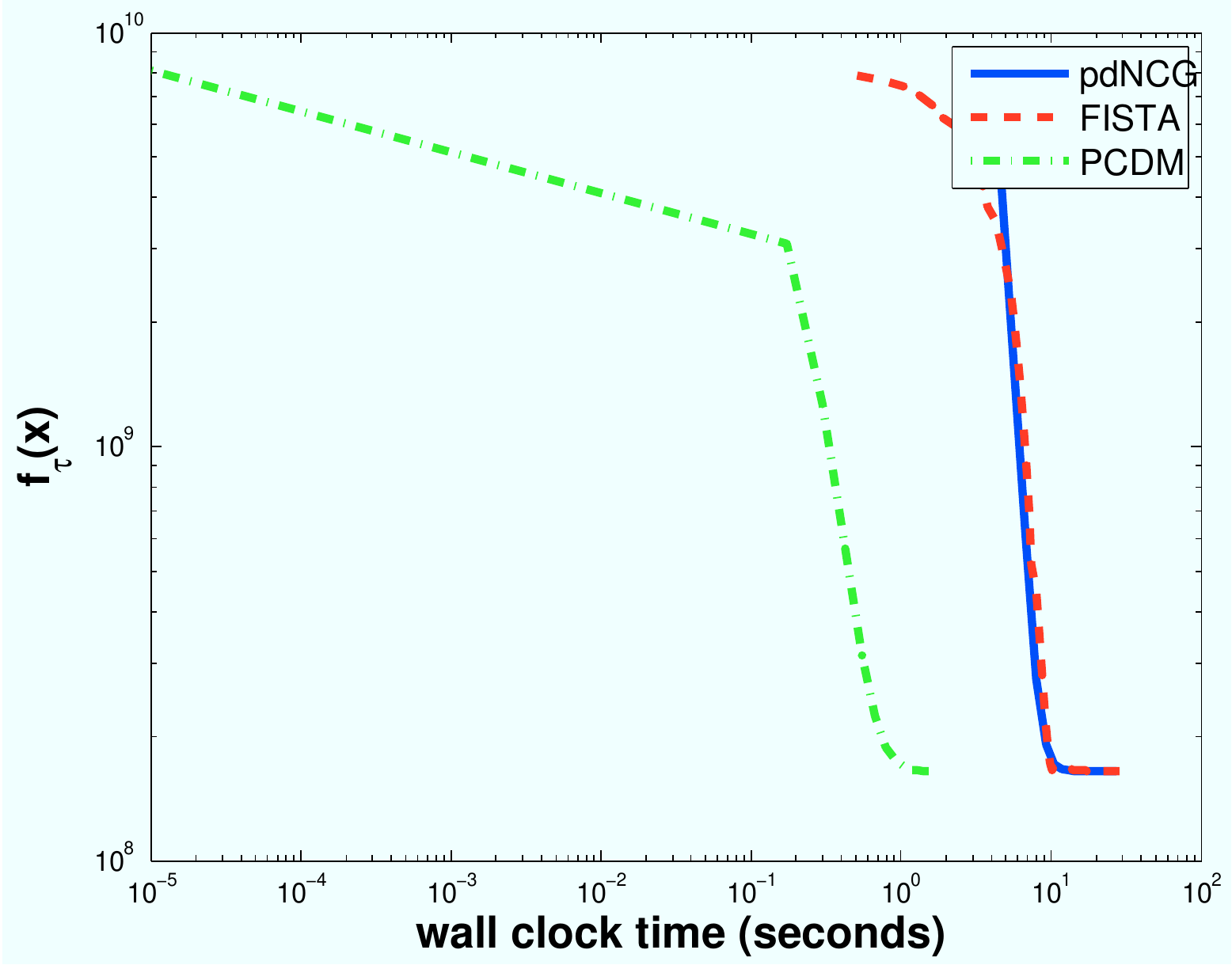}}
  \subfloat[$\kappa(A^\intercal A) = 1.0e$+$04$]{\label{fig2b}\includegraphics[scale=0.37]{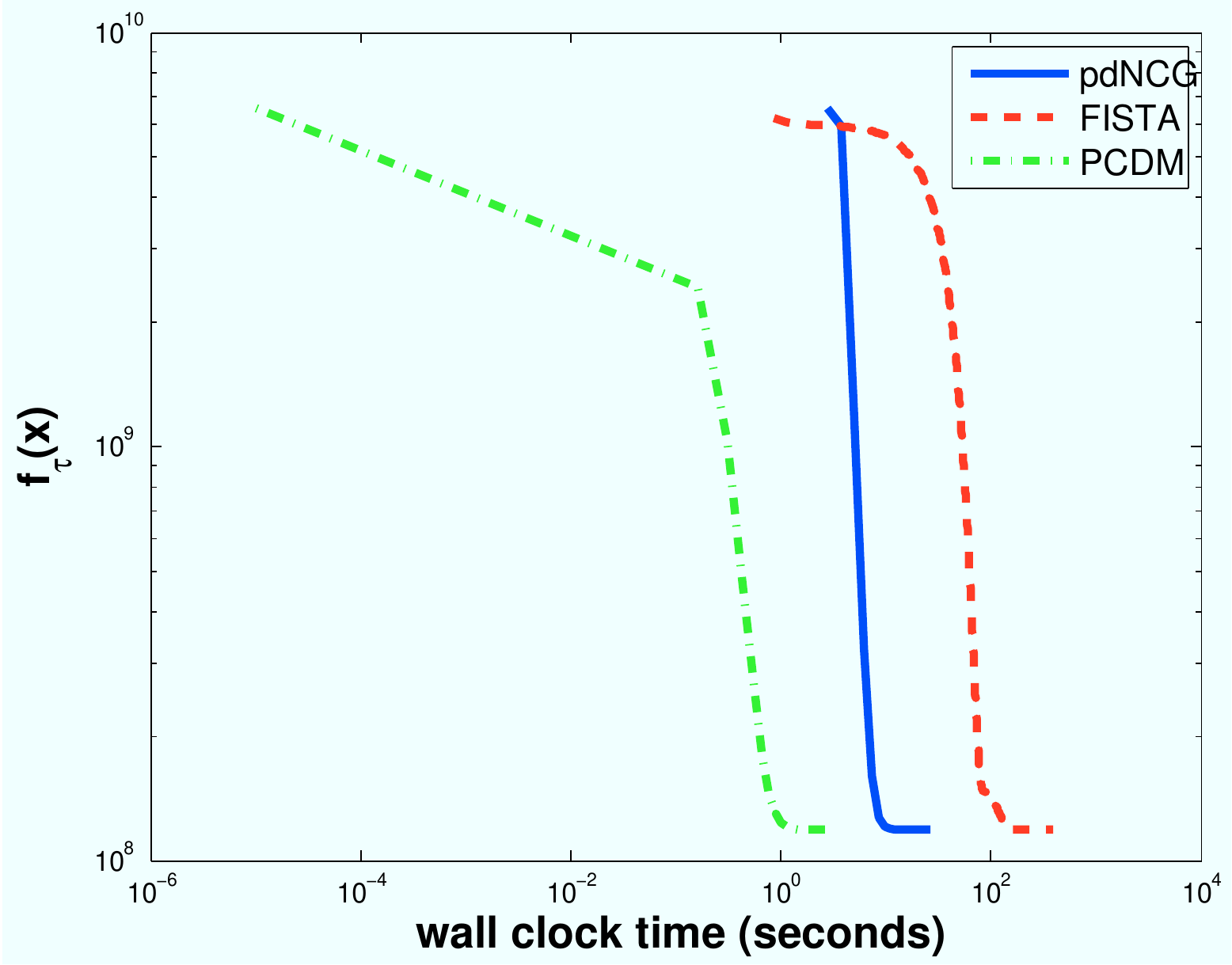}}
  \\
  \subfloat[$\kappa(A^\intercal A) = 1.0e$+$06$]{\label{fig2c}\includegraphics[scale=0.37]{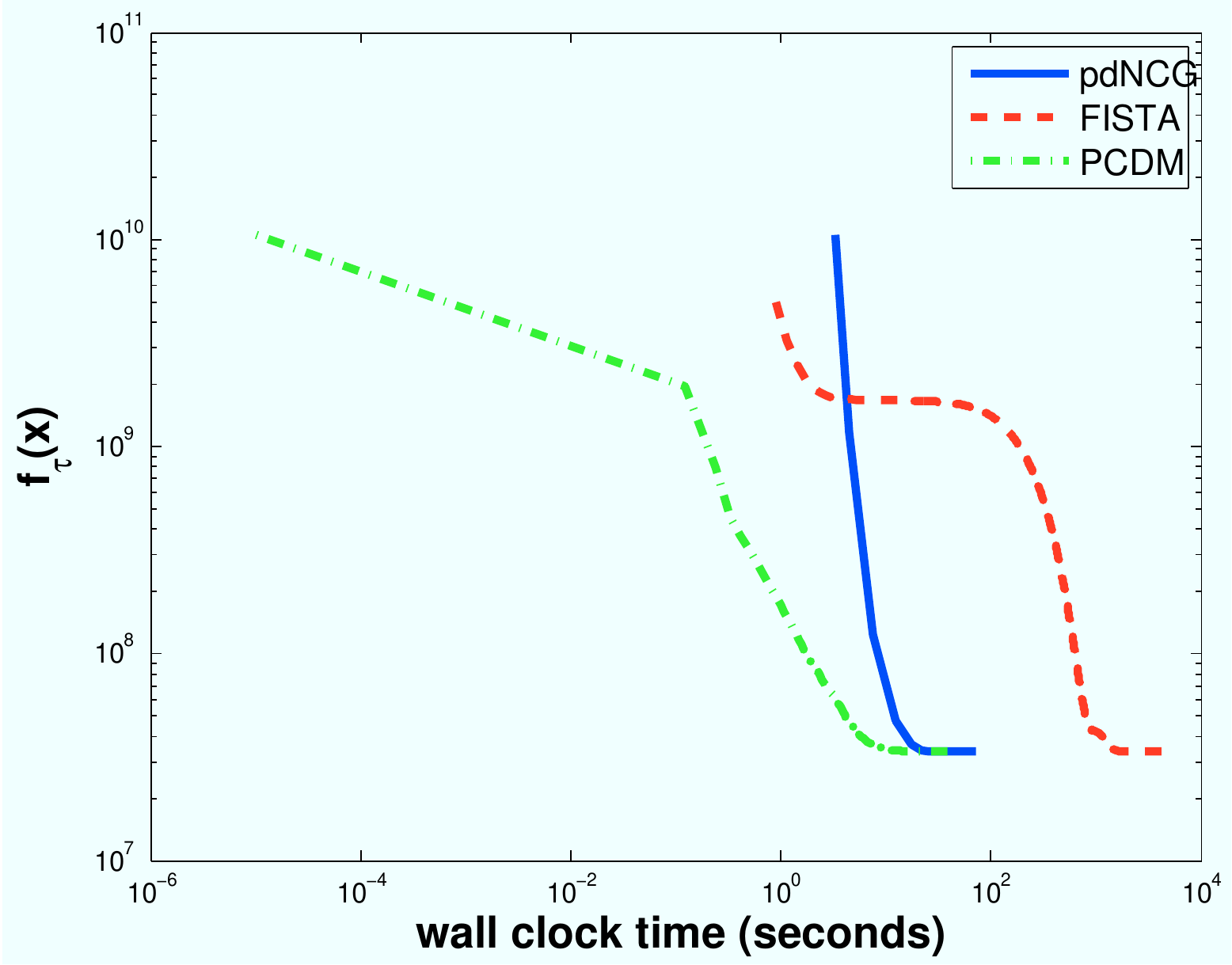}}
  \subfloat[$\kappa(A^\intercal A) = 1.0e$+$08$]{\label{fig2d}\includegraphics[scale=0.37]{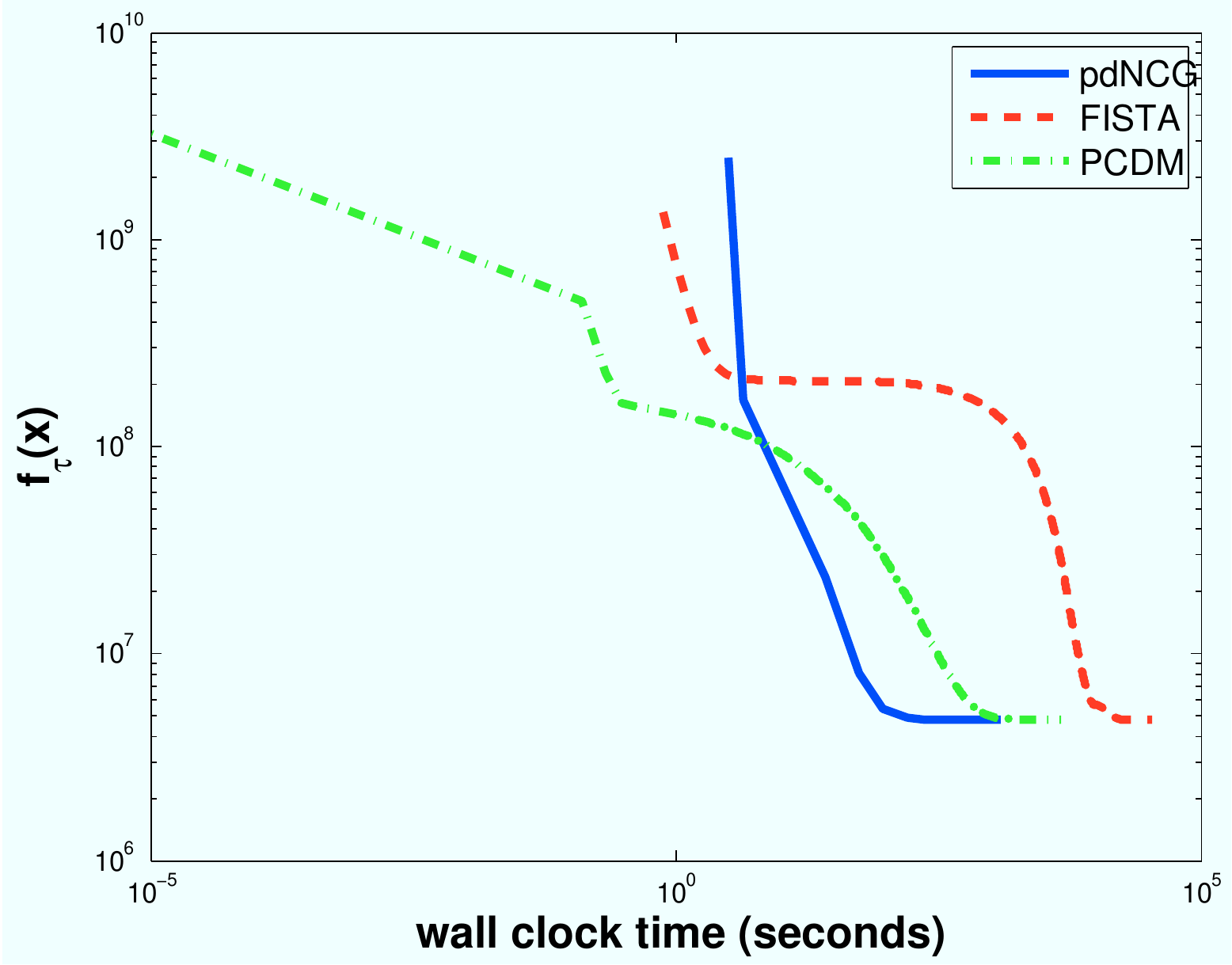}}
  \\
  \subfloat[$\kappa(A^\intercal A) = 1.0e$+$10$]{\label{fig2e}\includegraphics[scale=0.37]{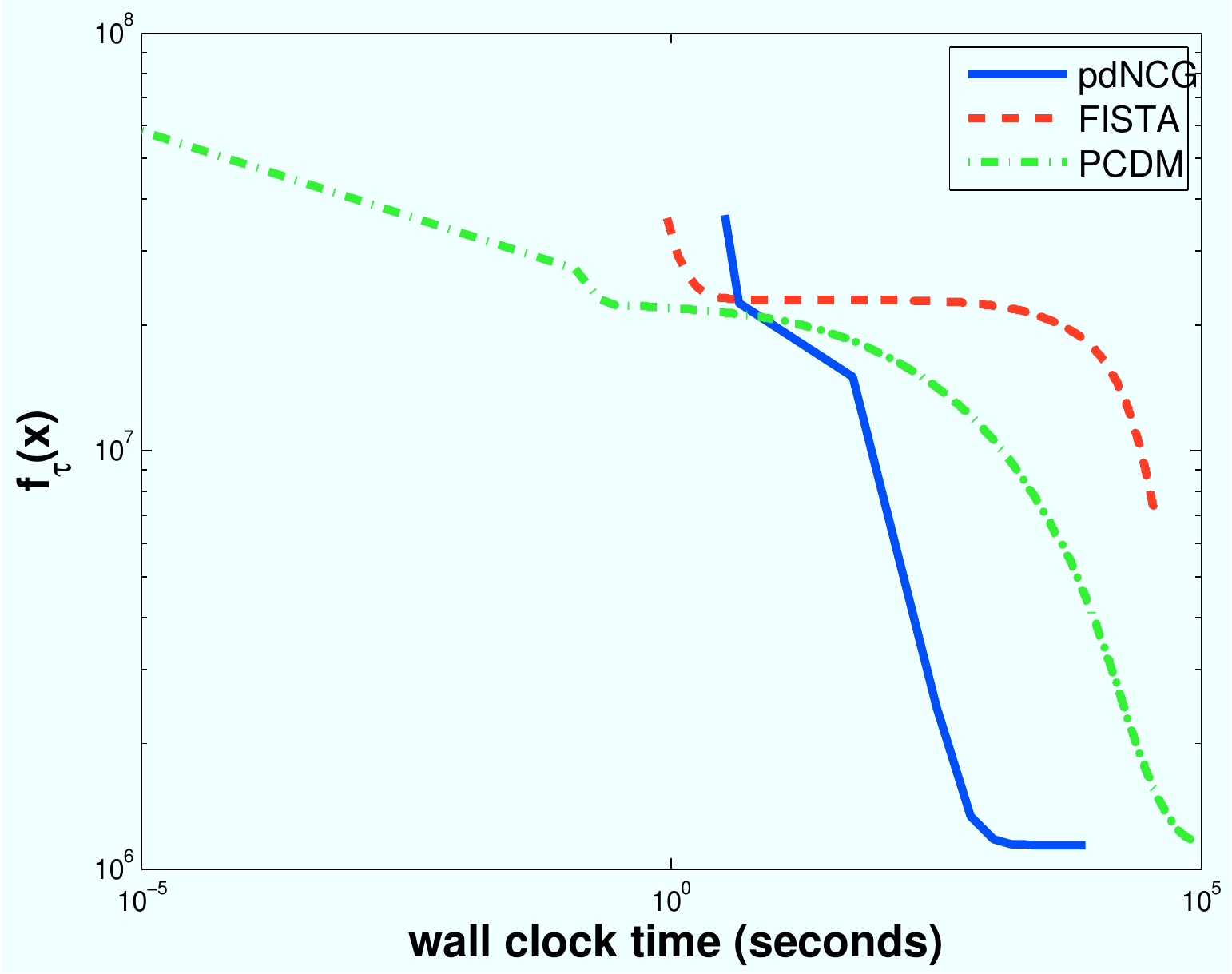}}
  \subfloat[$\kappa(A^\intercal A) = 1.0e$+$12$]{\label{fig2f}\includegraphics[scale=0.37]{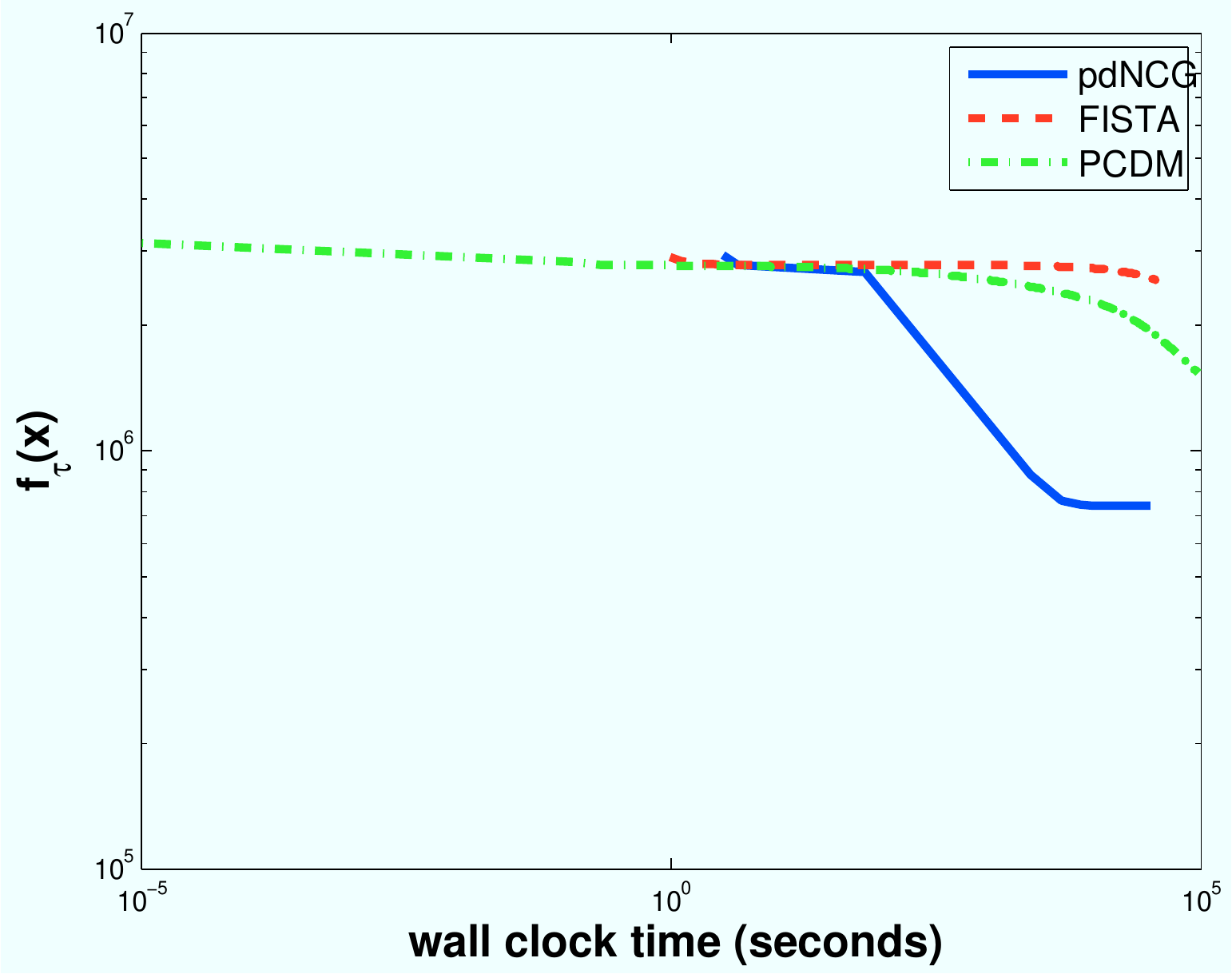}}
\caption{Performance of pdNCG, FISTA and PCDM on a synthetic sparse S-LS problem for increasing condition number of matrix $A^\intercal A$, denoted
as $\kappa(A^\intercal A)$. 
The axis are in log-scale. In this figure $f_\tau(x)$ denotes the objective value that was obtained by each solver}
\label{fig2}%
\end{figure}

\subsubsection{Synthetic sparse least-squares example: increasing dimensions}
In this experiment we present the performance of pdNCG, FISTA and PCDM as the number of variables $m$ increases. 
We generate three instances $(A,b,x^*)$, where $m$ takes values $2^{20}$, $2^{22}$ and
$2^{24}$.
Matrix $A$ has $n= 2 m$ rows, rank $m$ and the condition number of $A^\intercal A$ is $1.00e$\scalebox{.75}{$\mplus$}$08$.
Moreover, matrix $A$ is sparse, i.e., $\beta_{LS}= 2$, $nnz(A)/(mn)\approx 1.00e$-$06$, $3.00e$-$07$ and $5.00e$-$08$, respectively for each $m$.
The optimal solution $x^*$ has again approximately $s \approx 8.0e$-$03 m$ non-zero components.

The results of this experiment are presented in Figure \ref{fig3}. 
Observe that the required wall-clock time for pdNCG scaled similarly to the first-order methods FISTA and PCDM, despite being a second-order method.

\begin{figure}%
  \centering

  \subfloat[$m = 2^{20}$]{\label{fig3a}\includegraphics[scale=0.37]{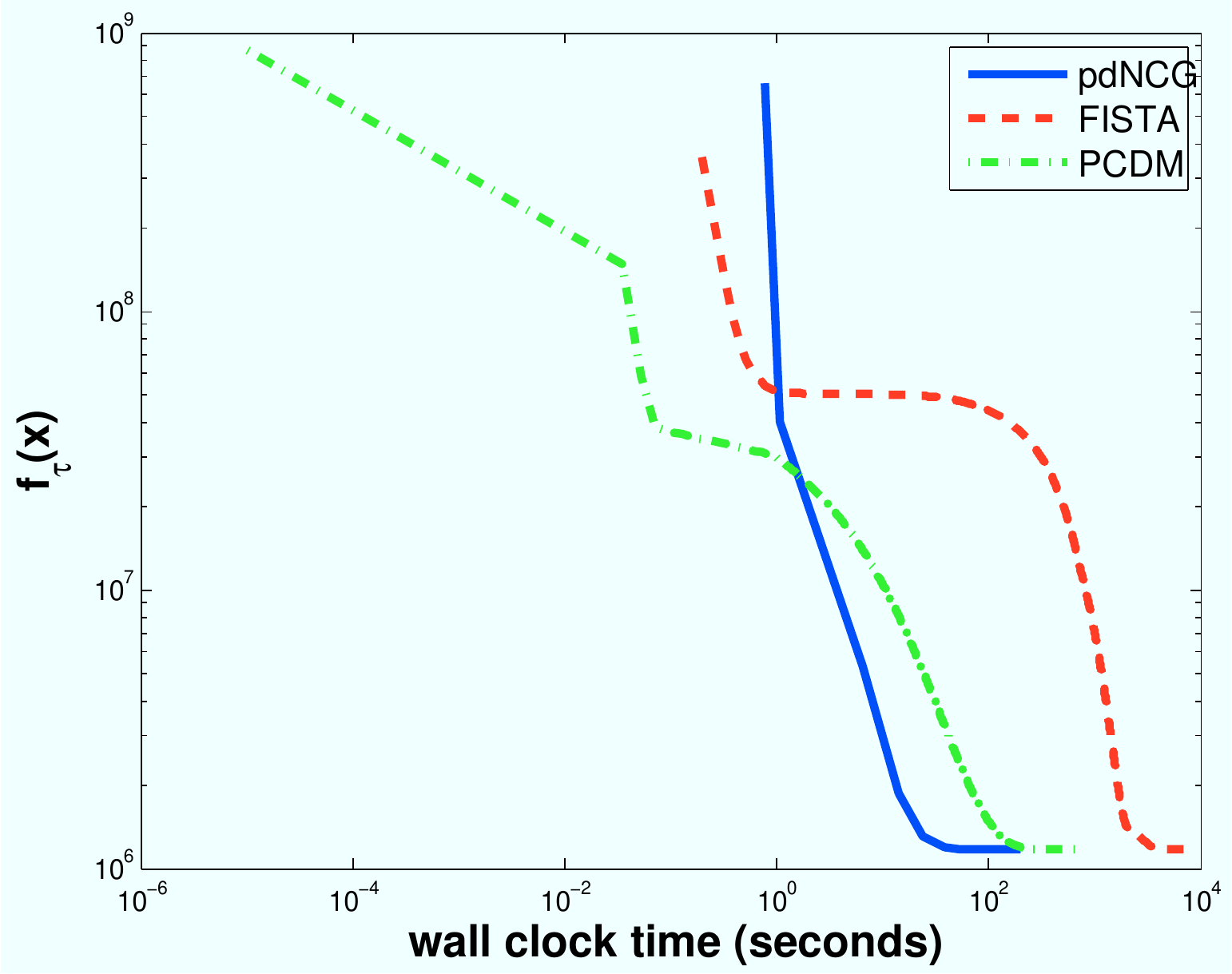}}
  \subfloat[$m = 2^{22}$]{\label{fig3b}\includegraphics[scale=0.37]{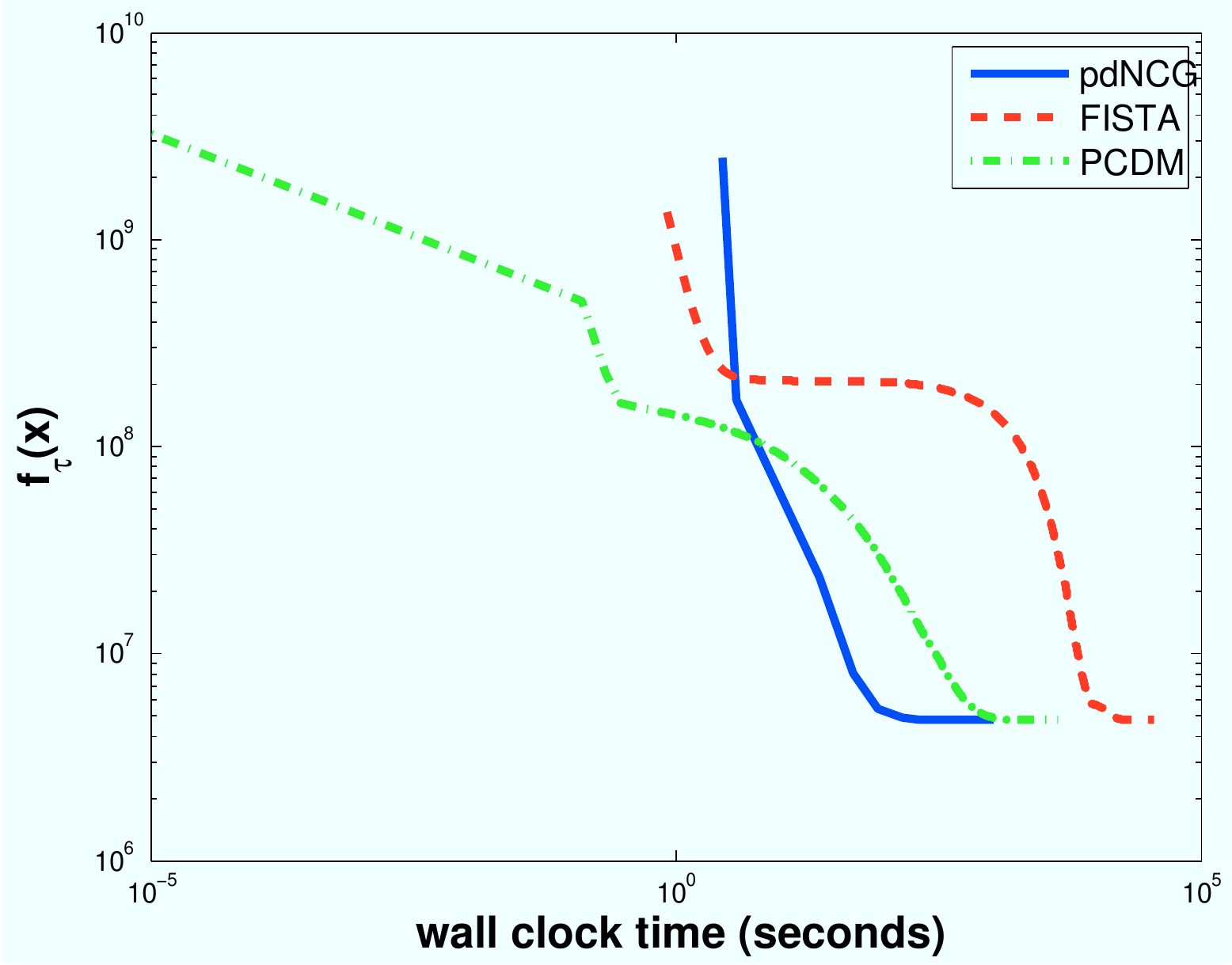}}
  \\
  \subfloat[$m = 2^{24}$]{\label{fig3c}\includegraphics[scale=0.37]{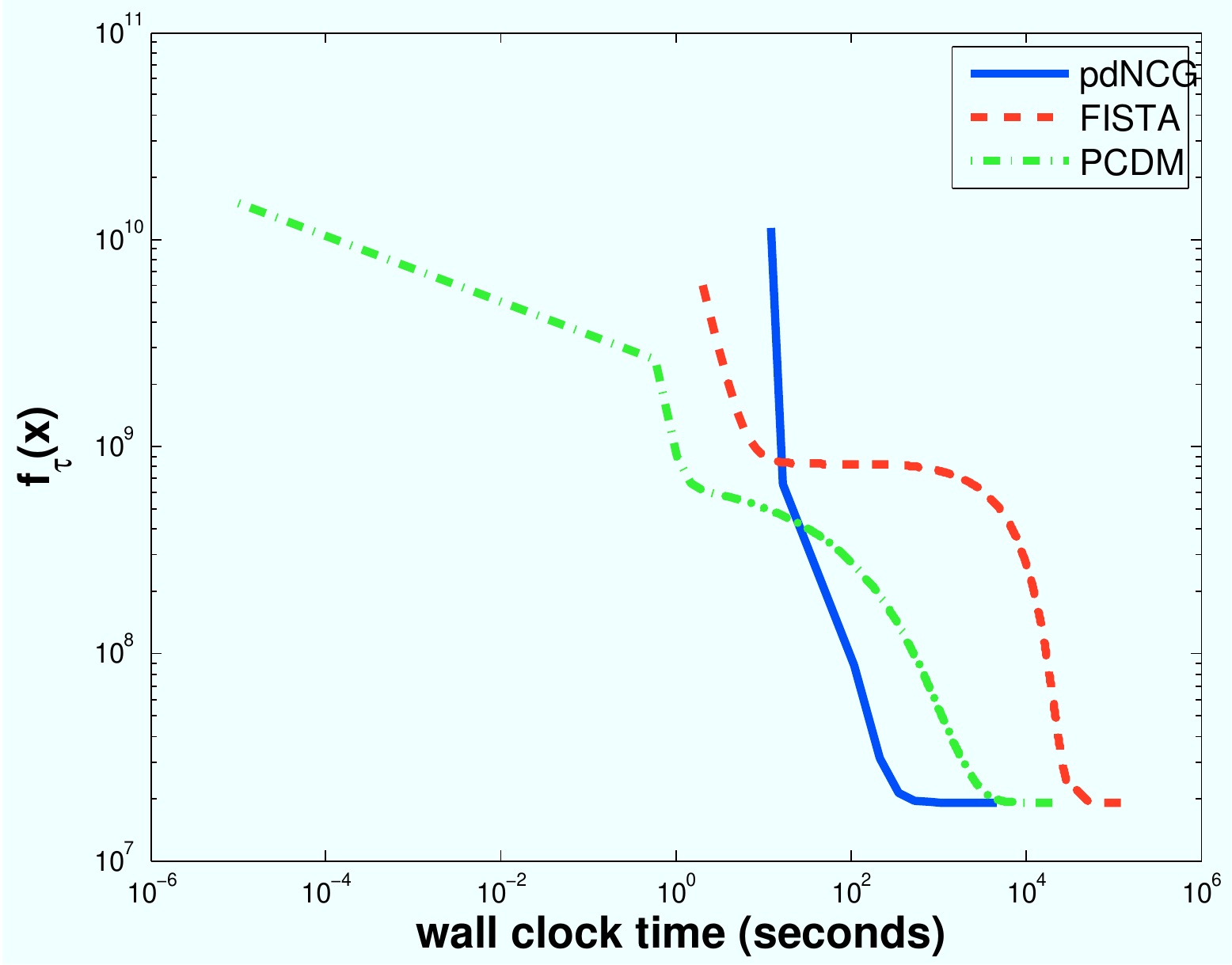}}
\caption{Performance of pdNCG, FISTA and PCDM on a synthetic sparse S-LS problem for increasing number of variables $m$. 
The axis are in log-scale. In this figure $f_\tau(x)$ denotes the objective value that was obtained by each solver}
\label{fig3}%
\end{figure}

\subsection{$\ell_1$-Regularized Logistic Regression}
In this subsection we compare pdNCG with FISTA and PCDM on six real world $\ell_1$-regularized LR problems. 
 For $\ell_1$-regularized LR the function $\varphi(x)$ in (\ref{prob1}) is set to
 \begin{equation*}
\varphi(x)=\sum_{i=1}^n\log(1+e^{-y_iw^\intercal x_i}),
\end{equation*} 
where $x_i\in\mathbb{R}^m$ $\forall i=1,2,\ldots,n$ are the training samples and $y_i\in\{-1,+1\}$ are the corresponding labels. Such problems are used for training a linear classifier $w\in\mathbb{R}^m$. 
Although in Linear Support Vector Machine (LSVM) literature there are more alternatives for function $\varphi(x)$, in this section we choose LR because it is second-order differentiable. For more details about support vector machine problems we refer the reader to \cite{yuanho}.

We present six $\ell_1$-regularized LR problems that are of large scale, sparse and partially or highly separable.
Exact information for these problems is given in Table \ref{LRprobs}. 
\begin{table}[t]
	\centering
	\caption{Properties of six $\ell_1$-regularized LR problems, which are used as benchmarks in this paper. The second and third columns show the number of training samples and features, respectively. 
	The fourth and fifth columns show the sparsity of matrix $X$ and the degree of partial separability $\beta_{LR}$ in \eqref{bd84}, respectively.
	The last column is the $\tau$ found using fivefold cross-validation}
	\begin{tabular}{lccccc}
		\textbf{Problem}& \textbf{$\boldsymbol n$}	& \textbf{$\boldsymbol m$}		& \textbf{$\mathbf{ nnz(X)/(mn)}$} & \textbf{$\boldsymbol \beta_{LR}$} &  $\boldsymbol \tau$\\
		\cite{realsim} real-sim		& $72,309$	& $20,958$	& $2.40e$-$02$ & $3484$   & $4.00e$-$02$\\
		\cite{rcv1}	 rcv1 & $20,242$     & $47,236$	& $1.60e$-$02$ & $980$ &$4.00e$-$02$ \\
		\cite{news20} news20	& $19,996$     & $1,355,191$	& $3.35e$-$04$ & $16423$ &$4.00e$-$02$ \\
		\cite{kdda} kdd (algebra) & $8,407,752$     & $20,216,830$	& $1.79e$-$06$ & $85$ &$2.00e$-$00$ \\
		\cite{kdda} kdd (br. to alg.) & $19,264,097$     & $29,890,095$	& $9.83e$-$07$ & $75$ &$2.00e$-$00$ \\
		\cite{webspam} webspam	& $350,000$     & $16,609,143$	& $2.24e$-$04$ & $46991$ &$4.00e$-$02$ 
	\end{tabular}
	\label{LRprobs}
\end{table}
In this table, matrix $X\in\mathbb{R}^{n\times m}$ has the training samples in its rows, the fourth column shows the sparsity of matrix $X$, where $nnz(X)$ is the number of non-zero components in $X$. The fifth column shows the degree of partial separability, which is defined as
\begin{equation}\label{bd84}
\beta_{LR} := \max_{j\in[1,2,\ldots,n]} \|X_j\|_0,
\end{equation}
where $X_j$ is the $j^{th}$ row of matrix $X$. 
The last column shows the $\tau$ that gave the classification with the highest accuracy after performing a fivefold cross validation over various $\tau$ values,
as proposed in \cite{svmguid}. The calculated values $\tau$ resulted for all problems in more than $90\%$ classification accuracy.
All problems in Table \ref{LRprobs} can be downloaded from the collection of LSVM problems in \cite{CC01a}.
Notice that for most of the problems in Table \ref{LRprobs}, $m < n$, which means that the problems are not strongly-convex everywhere 
as assumed in \eqref{bd15}. However, the problems in Table \ref{LRprobs} have a unique solution, which implies that they are strongly-convex 
locally to the optimal solution. We choose to solve these instances because the problems for which $m > n$ in collection \cite{CC01a} are small
scale, hence, possible numerical experiments might not provide significant insight into the behaviour of the methods. It is important to mention that 
all implementations of the compared methods can handle such cases without any modification.

The results of the comparison among the solvers pdNCG, FISTA and PCDM are shown in Figure \ref{fig4}. Notice in Subfigure \ref{fig4d} that for pdNCG 
the objective function $f_\tau(x)$ seems not to decrease always monotonically. This behaviour might have occurred because
backtracking line-search can terminate before the condition in Step $4$ of pdNCG is satisfied, if the maximum number of backtracking iterations 
is exceeded. 
%

\begin{figure}[htp]
  \centering

  \subfloat[real-sim]{\label{fig4a}\includegraphics[scale=0.37]{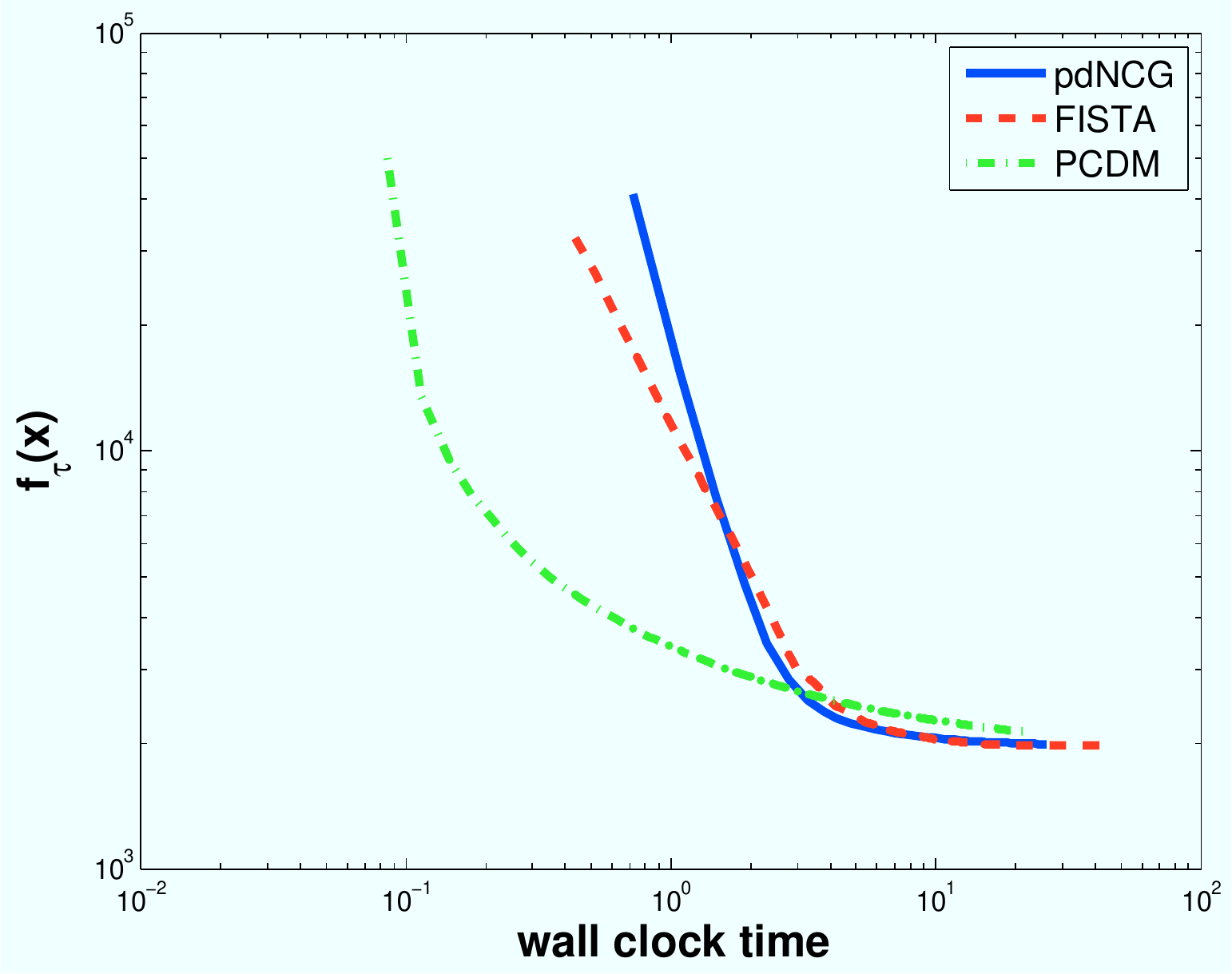}}
  \subfloat[rcv1]{\label{fig4b}\includegraphics[scale=0.37]{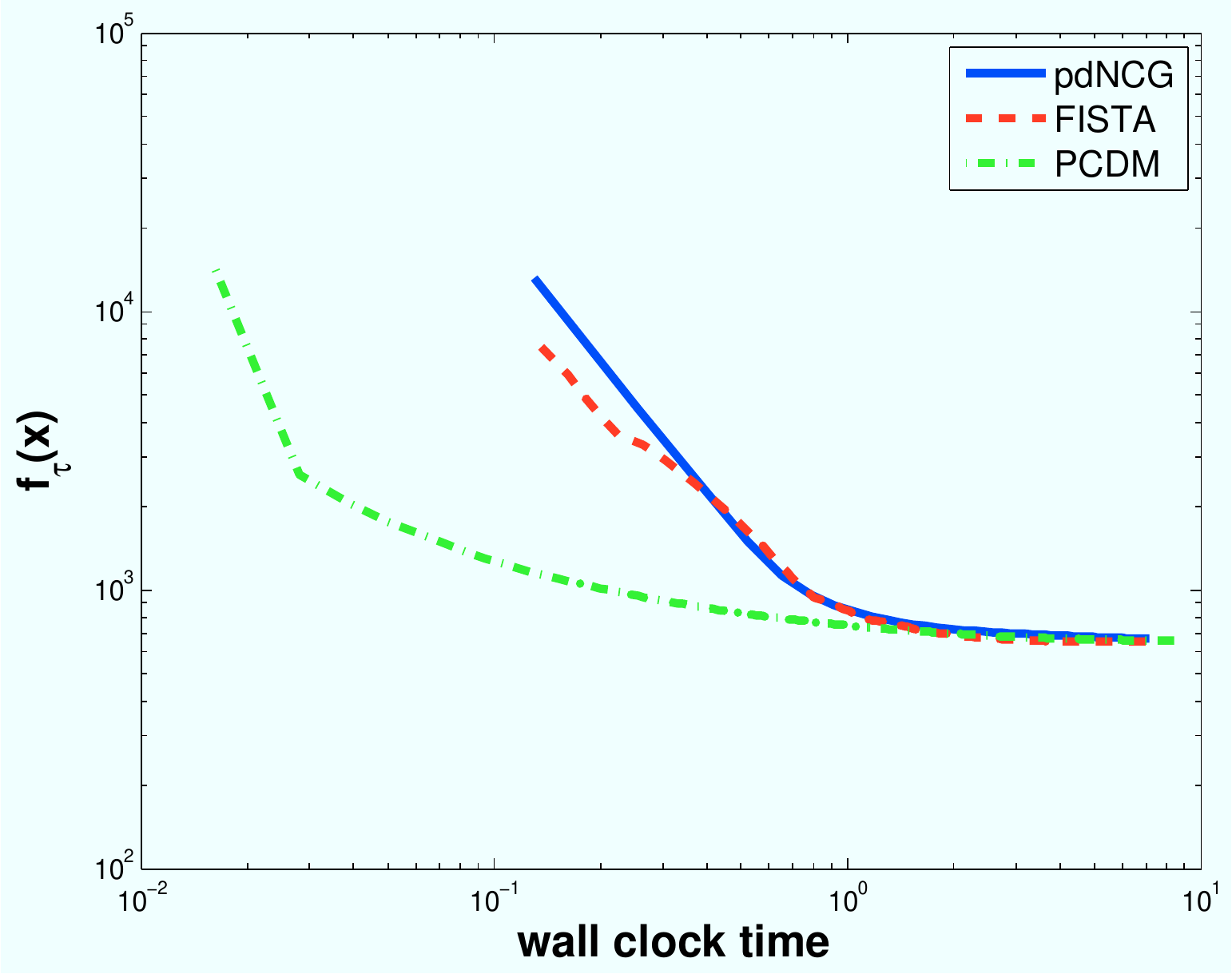}}
  \\
  \subfloat[news20]{\label{fig4c}\includegraphics[scale=0.37]{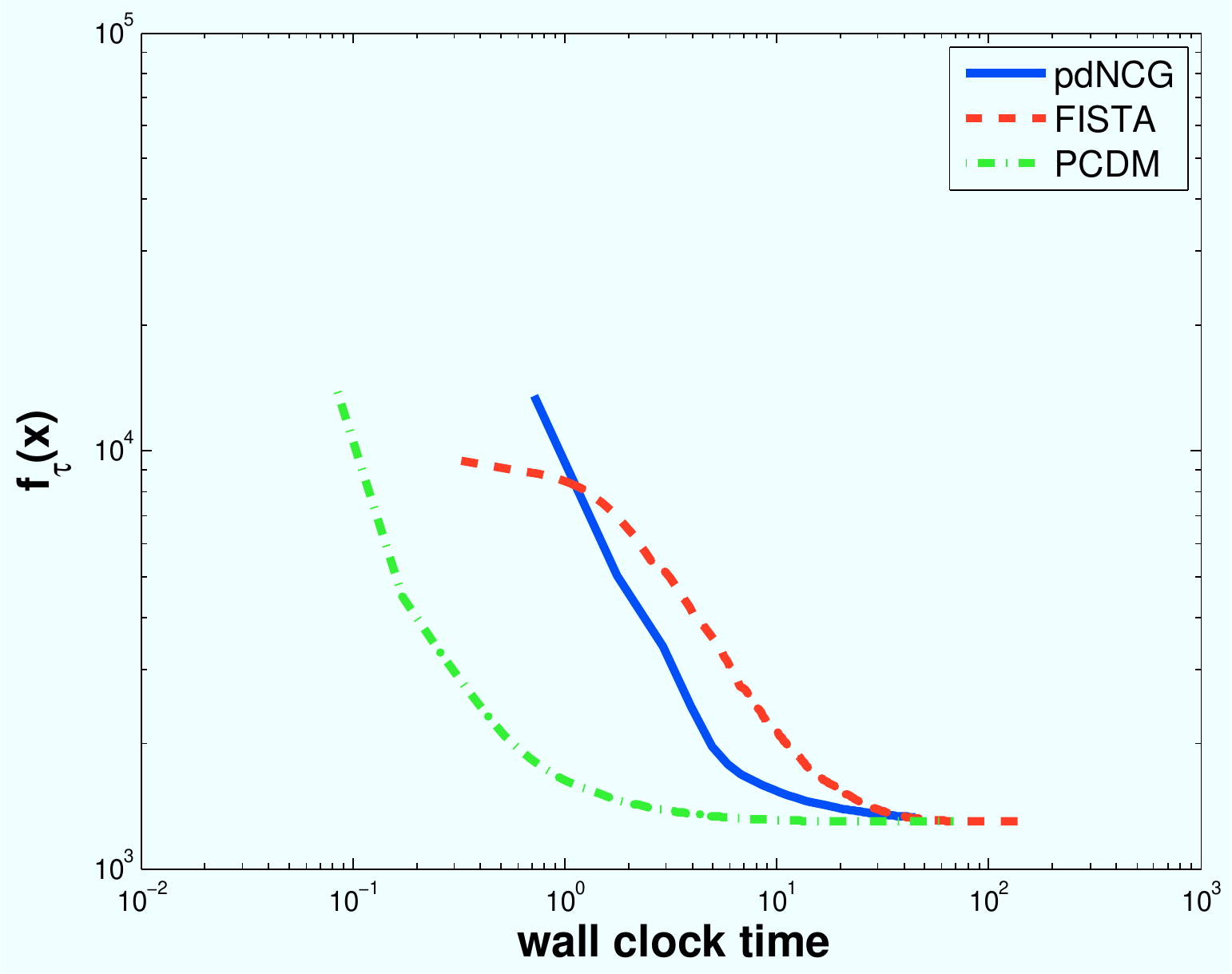}}
  \subfloat[kdd (algebra)]{\label{fig4d}\includegraphics[scale=0.37]{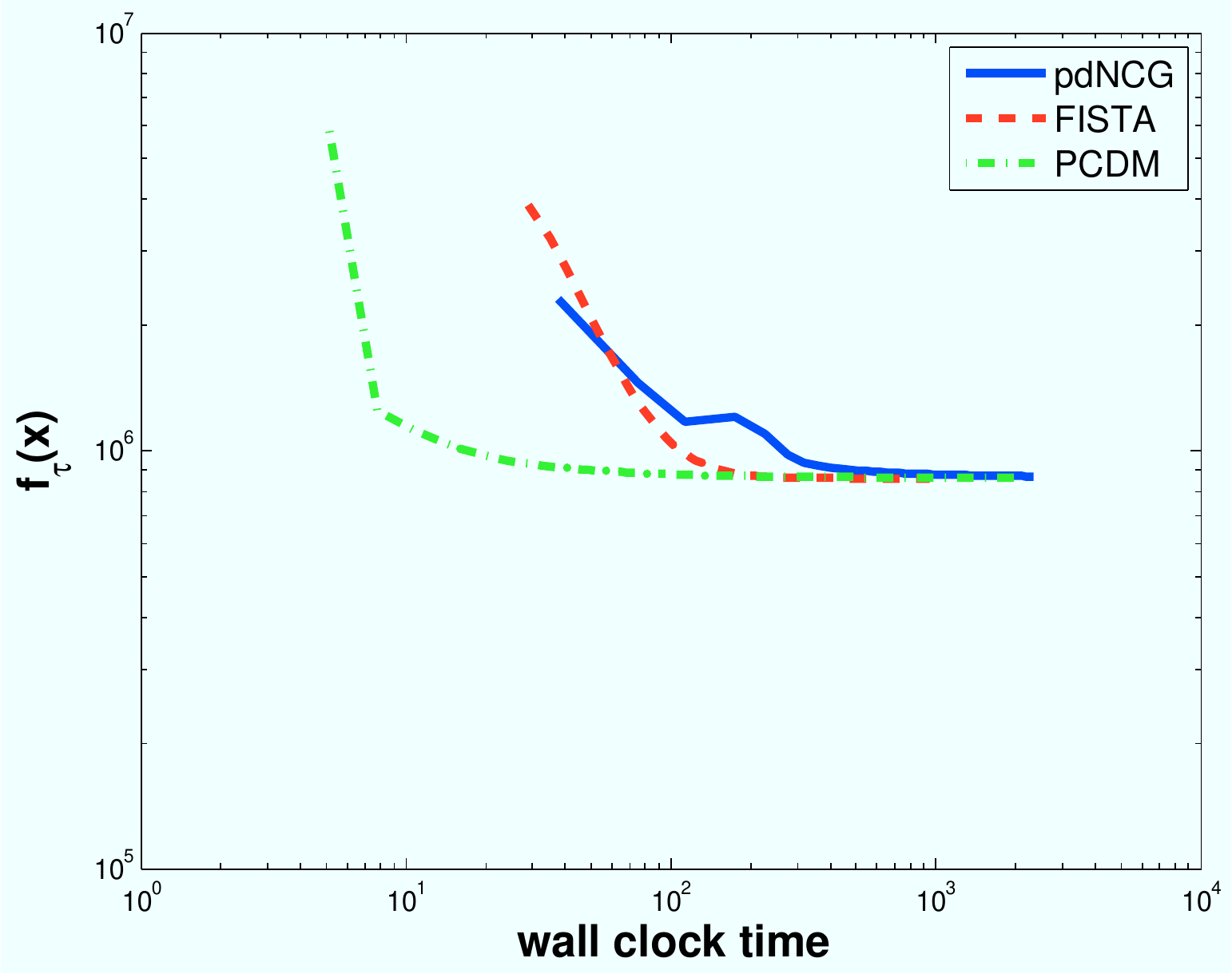}}
    \\
  \subfloat[kdd (bridge to algebra)]{\label{fig4e}\includegraphics[scale=0.37]{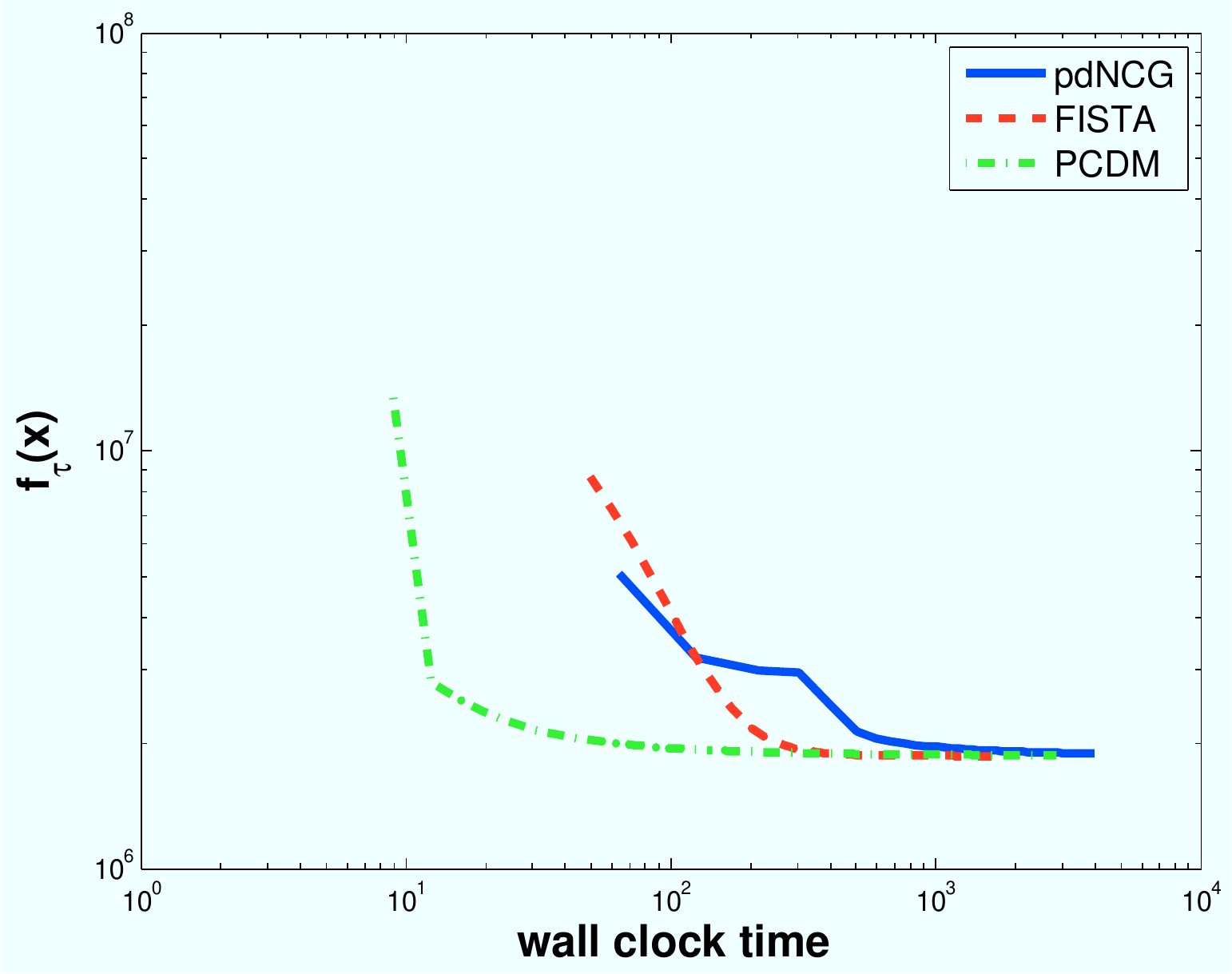}}
  \subfloat[webspam]{\label{fig4f}\includegraphics[scale=0.37]{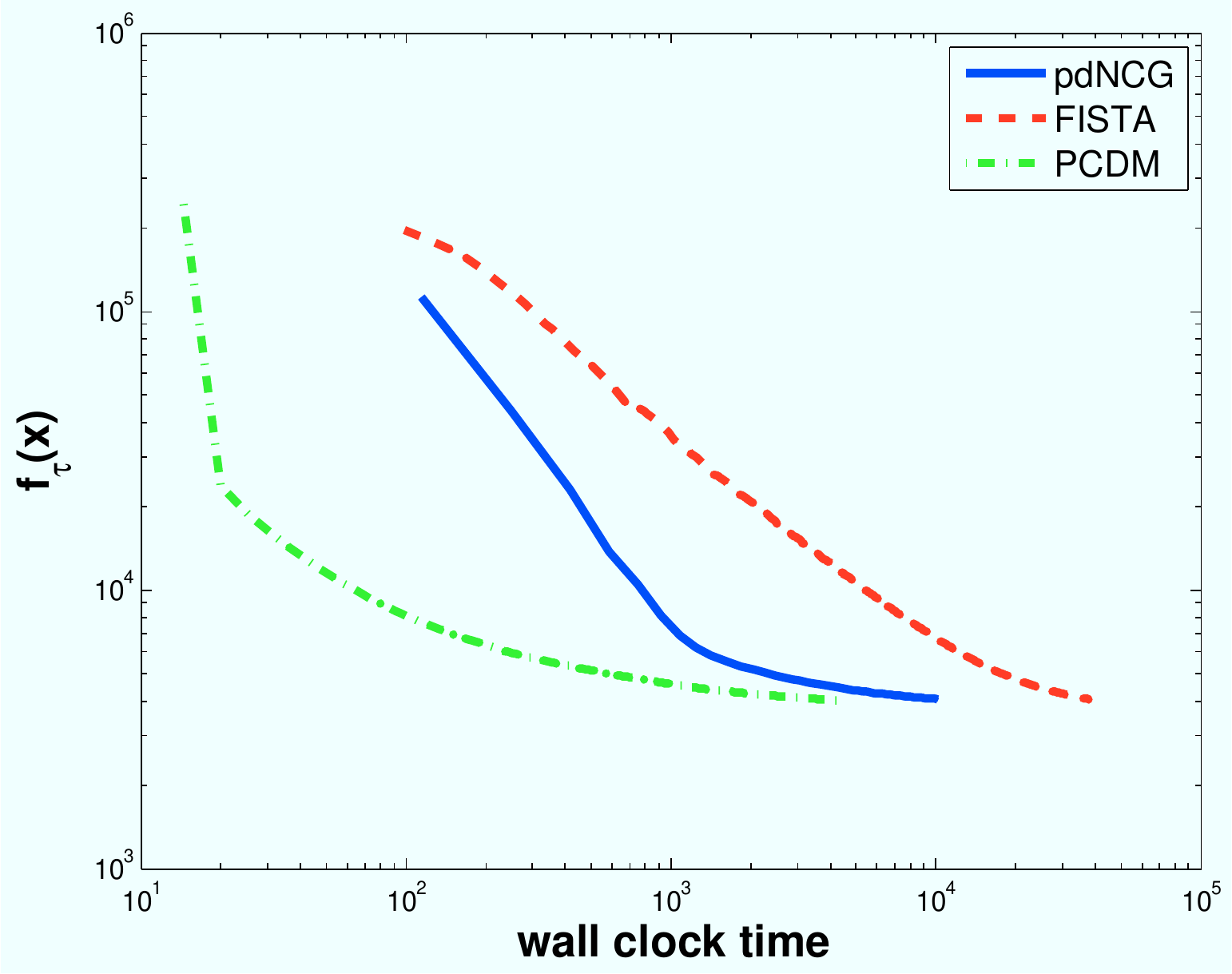}}
\caption{Comparison on $\ell_1$-regularized LR problems for pdNCG, FISTA and PCDM.}
\label{fig4}
\end{figure}

\section{Conlcusion}\label{sec:conclusion} 
In this paper we have studied an inexpensive but still robust primal-dual Newton-CG (pdNCG) method. 
The proposed method is developed for the solution of $\ell_1$-regularized problems, which might display
some degree of ill-conditioning; that is, display noticeable differences of the magnitude of eigenvalues. 
For such problems it is crucial that the methods capture information from the second-order derivative.
We have given synthetic sparse least-squares examples 
and six real world machine learning problems that satisfy the previous criteria and we provided computational evidence that on these problems the proposed method is efficient. 
An implementation of pdNCG and scripts that reproduce the numerical experiments can be downloaded from
\centerline{\url{http://www.maths.ed.ac.uk/ERGO/pdNCG/}.}

Finally, we have shown that by using the property of CG
described in Lemma \ref{lem:15}, the convergence analysis of pdNCG can be performed in a variable metric, which is defined based on approximate second-order 
derivatives. The variable metric opens the door for a tight convergence analysis of pdNCG, which includes global and local convergence rates, explicit definition of fast local convergence region
and worst-case iteration complexity.
 
\bibliographystyle{plain}
\bibliography{KFandJG.bib}

\end{document}